\documentclass{article}
\usepackage[T1]{fontenc}
\usepackage[utf8]{inputenc}
\usepackage{amsfonts}
\usepackage{amsmath}
\usepackage{mathtools}
\usepackage{amsthm}
\usepackage{dsfont}
\usepackage{enumerate}
\usepackage{color}
\usepackage[margin=1.35in]{geometry}
\usepackage{todonotes}
\usepackage[normalem]{ulem}
\usepackage{authblk}
\usepackage{mathscinet}

\usepackage[
backend=biber,
style=alphabetic,
]{biblatex}
\usepackage{hyperref}
\usepackage{xurl}
\hypersetup{breaklinks=true}

\usepackage{soul,xcolor}

\newcommand{\R}{\mathbb{R}}
\newcommand{\N}{\mathbb{N}}
\newcommand{\E}{\mathbb{E}}
\newcommand{\Prob}{\mathbb{P}}

\newcommand{\dd}{\,\mathrm{d}}
\DeclareMathOperator{\Span}{span}
\DeclareMathOperator{\Id}{Id}

\newtheorem{theorem}{Theorem}[section]
\newtheorem{definition}[theorem]{Definition}
\newtheorem{lemma}[theorem]{Lemma}
\newtheorem{assumption}[theorem]{Assumption}
\newtheorem{remark}[theorem]{Remark}
\newtheorem{proposition}[theorem]{Proposition}

\title{No blow-up by nonlinear It\^o noise for the Euler equations}
\author{Marco Bagnara\footnote{Classe di Scienze, Scuola Normale Superiore, Piazza dei Cavalieri 7, 56126 Pisa, Italy; \href{mailto:marco.bagnara@sns.it}{marco.bagnara@sns.it}}, Mario Maurelli\footnote{Dipartimento di Matematica, Universit\`a di Pisa, Largo Bruno Pontecorvo 5, 56127 Pisa, Italy; \href{mailto:mario.maurelli@unipi.it}{mario.maurelli@unipi.it}}, Fanhui Xu\footnote{Department of Mathematics, Union College, 807 Union St, Schenectady, NY 12308, USA; \href{mailto:fanhuixu@math.harvard.edu}{xuf2@union.edu}}}
\date{}

\begin{document}

\maketitle

\begin{abstract}
By employing a suitable multiplicative It\^o noise with radial structure and with more than linear growth, we show the existence of a unique, global-in-time, strong solution for the stochastic Euler equations in two and three dimensions. More generally, we consider a class of stochastic partial differential equations (SPDEs) with a superlinear growth drift and suitable nonlinear, multiplicative It\^o noise, with the stochastic Euler equations as a special case within this class. We prove that the addition of such a noise effectively prevents blow-ups in the solution of these SPDEs.
\end{abstract}

\section{Introduction}

In this paper, we consider the incompressible Euler equations in a $d$-dimensional bounded domain ($d=2,3$) perturbed by a nonlinear, multiplicative It\^o noise:
\begin{align}\label{EulerSto}
    du +\mathcal{P}[u\cdot \nabla u] dt = \sigma(u)dW_t,
\end{align}
where $u(t,x,\omega)\in \R^d$ is an unknown velocity field, $\mathcal{P}$ denotes the Leray projector, $W$ is a one-dimensional Brownian motion, and $\sigma$ is an appropriate coefficient exhibiting superlinear growth, for instance, $\sigma(u) = \Vert u\Vert_{W^{1,\infty}}^{1/2+\epsilon}\,u$. Our primary result establishes the global-in-time existence of smooth solutions. Notably, for $d=3$, this result is generally unknown without noise, and counterexamples exist in several function spaces (see examples below).
Thus, the incorporation of noise in our model averts potential blow-ups. Moreover, we provide a broad framework to ensure the absence of blow-ups in the solutions of hyperbolic-type SPDEs.

\textbf{Deterministic and stochastic Euler equations:} The deterministic Euler equations describe the motion of an incompressible, non-viscous fluid. Local well-posedness has been established in H\"older spaces $C^{m,\alpha}$, where $m\geq 1$ is an integer and $0<\alpha<1$ (see~\cite{Lichtenstein1925, gunther1927motion}), as well as in Sobolev spaces $W^{s,p}$, with $s > d/p + 1$, $1<p<\infty$, and the spatial dimension $d \ge 2$ (see~\cite{EbiMar1970,Kato1988CommutatorEA}). In two-dimensional cases, global well-posedness has been observed in H\"{o}lder spaces (\cite{Wolibner1933}) and Sobolev spaces $W^{1,p}$ (\cite[Theorem~6.1]{Yud1963}). However, in three dimensions, the global existence of smooth solutions continues to be an open problem. Remarkable recent advancements have revealed solutions exhibiting blow-up phenomena for closely related problems. In their groundbreaking work \cite{Bourgain2014StrongIO}, Bourgain and Li proved that the 3D Euler equations are ill-posed in $C^m(\mathbb R^d)$ and $C^{m-1,1}(\mathbb R^d)$ for all integers $m\geq 1$. They later expanded this conclusion in \cite{bourgain2015strong,Bourgain2021} to encompass the borderline spaces $W^{d/p+1,p}(\mathbb R^d)$ and Besov spaces $B^{d/p+1}_{p,q}(\mathbb R^d)$ for any $1\leq p<\infty$, $1<q\leq \infty$. Elgindi presented an example of explosive vorticity of the $C^{1,\alpha}(\mathbb R^d)$-solution for some $\alpha>0$ in \cite{Elg2021}. Chen and Hou proved in \cite{CheHou2022} blow-ups for smooth solutions of the 3D Euler equations in the axisymmetric case, with no-flow boundary conditions in the radial variable and periodic boundary conditions in the height variable; see also the work \cite{LuoHou2014} by Luo and Hou showing numerical evidence of blow-up in the axisymmetric case.

Although in our paper we focus on spatially smooth solutions, we briefly recall that irregular solutions can give rise to anomalous dissipation and non-uniqueness in both the Euler and Navier-Stokes equations. Notable works in this direction include (among many others) those by De Lellis and Sz\'{e}kelyhidi \cite{DeLSze2014} and Isett \cite{Ise2018} on the 3D Euler equations, Vishik's works on the 2D Euler equations \cite{Vis2018_1,Vis2018_2}, as well as Albritton, Bru\'e, and Colombo's work on the forced 3D Navier-Stokes equations \cite{AlbBruCol2022}.

Concerning the stochastic case, the study of the stochastic Euler equations began with the two-dimensional case, yielding global existence and global well-posedness results across various function spaces, see
\cite{doi:10.1080, BesFla1999, BrzPes2001, 10.2307/2667278, GlaVic2014, doi:10.1137,MikVal2000} 
and references therein. In three-dimensional cases, among other papers, Kim \cite{Kim2009} derived a local strong $H^{s}(\mathbb{R}^3)$-solution under additive noise conditions, provided $s>5/2$. Glatt-Holtz and Vicol \cite{GlaVic2014} showed local well-posedness results for the 3D stochastic Euler equations driven by a wide range of multiplicative noise. They proved that the solution is global with high probability when the multiplicative noise is linear. Euler equations with transport noise have also been studied, including, among others, Brze\'{z}niak, Flandoli, and Maurelli's work on the 2D case \cite{BrzFlaMau2016}, and Crisan, Flandoli, and Holm's work on the 3D case \cite{CriFlaHol2019}. We briefly mention that a lot of works have been devoted to stochastic Navier-Stokes equations (e.g. \cite{FlaGat1995,HaiMat2006,BiaFla2020,FerZan2018}).

\textbf{Regularization and prevention of blow-up by noise in finite dimensions:} The lack of a global well-posedness result in the deterministic 3D Euler equations has lead to the question of whether noise can restore well-posedness. This concept is known as regularization by noise. A well-understood example in this regard is the case of ODEs with non-smooth drifts, where the addition of an additive noise can bring well-posedness, see for example the works by Krylov and R\"ockner~\cite{krylov2005strong}, Fedrizzi and Flandoli~\cite{fedrizzi2011pathwise}, Catellier and Gubinelli \cite{catellier2016averaging}, among many others. This regularization result for ODEs can be translated into a regularization effect, by transport noise, for linear transport equations with non-smooth drift, see the seminal work \cite{flandoli2010well} by Flandoli, Gubinelli, and Priola. A motivation behind these works is that velocity fields arising from fluid models are often irregular, hence a linear transport equation with a non-smooth drift is an appropriate initial model for possible regularization phenomena in fluid dynamics.

Here, we delve into another regularization effect of noise, considering ODEs whose drifts have superlinear growth, and investigating whether suitable noise can prevent solutions from blowing up. For finite-dimensional SDEs, no-blow-up criteria and results, as well as stabilization by noise results, are known since at least the works of Khasminskii. For example, for linear systems, a linear multiplicative It\^o noise can provide a damping effect which turns unstable equilibria into stable ones, see e.g. the example at the end of \cite[Section 5.3]{khasminskii2011stochastic} and the stabilization criterion in \cite{Has1967}. A multiplicative It\^o noise with superlinear growth can enhance this damping effect, suppressing blow-up for ODEs, as long as the growth of the diffusion coefficient is strong enough with respect to the growth of the drift. This has been shown in several works, for example Appleby, Mao, and Rodkina \cite{4484185}, Gard \cite{gard1988introduction} (see also the Khasminskii's no blow-up criterion in \cite{Has1960}). For Stratonovich noise, a no blow-up result for ODEs also holds but with multidimensional noise and suitably chosen superlinear growth coefficient, see Maurelli \cite{maurelli2020non}. A classical method to show these no blow-up results is the Lyapunov function method (see again the book \cite{khasminskii2011stochastic}): one studies the evolution of a concave function $V$, such as $|x|^{1/2}$ or $\log|x|$, of the norm of the solution $X$ to the SDE and exploits the negative contribution in the second order term (hence coming from the noise) in the It\^o formula for $V(X)$. Let us also mention another class of no blow-up results, where the ODE may explode along ``rare'' directions, and the noise, which might be simply additive, prevents blow-ups by steering the solution away from these explosive directions. Examples include the works by Scheutzow \cite{doi101080Sch}, Herzog and Mattingly \cite{10.1214/EJP.v20-4047}, Athreya, Kolba, and Mattingly \cite{Athreya2012}. Finally we recall the celebrated stabilization result \cite{ArnCraWih1983} on stabilization of linear system by linear Stratonovich noise, which exploits Lyapunov exponents and tools from ergodic theory and whose stabilization mechanism is based on random rotations rather than damping.

\textbf{No blow-up by noise in infinite dimension:} Here we consider the problem of no blow-up by noise for Euler equations and more generally for hyperbolic-type SPDEs. Our work goes in the direction of extending the finite-dimensional results with multiplicative superlinear It\^o noise to the infinite-dimensional setting. We show that, with full probability, this kind of noise prevents solutions to certain SPDEs from blowing-up. Our approach is based on Lyapunov function, as from the finite-dimensional case, together with tools from the monotone settings for SPDEs (see e.g. \cite{LiuRoc2015}). Let us recall some related literature on the problem. Concerning linear noise, the damping effect of a large linear multiplicative It\^o noise has been used to show no blow-up with high probability for SPDEs, in particular for 3D Euler equations, by Glatt-Holtz and Vicol \cite{GlaVic2014}, and for nonlinear Schr\"odinger equation, by Barbu, R\"ockner and Zhang \cite{BarRocZha2017}. Most of the works on SPDEs with superlinear noise are more recent, we cite some of them without claim of completeness. Alonso-Or\'an, Miao, and Tang in \cite{ALONSOORAN2022244} used the Lyapunov function method to establish a no-blow-up result for a one-dimensional transport-type PDE with a nonlinear diffusion coefficient. While finalizing our paper, we became aware of the papers  \cite{RTW2020} and \cite{TW2022}. Ren, Tang, and Wang employed Lyapunov fucntions in \cite{RTW2020} for SPDEs characterized by a (possibly irregular and distribution-dependent) drift with superlinear growth and apply it to certain transport-type SPDEs. To our knowledge, the paper \cite{TW2022} by Tang and Wang gives the most general setting and the one most closely related to ours. The authors revisited the Lyapunov function criterion from \cite{RTW2020} and extended its application to a broad class of singular SPDEs, including the stochastic Euler and Navier-Stokes equations. They specifically noted that ``a fast enough growth of the noise coefficient will kill the growth of other terms so that the non-explosion is ensured''. With respect to the Euler equations, they derived a no-blow-up condition for the growth of the diffusion coefficient, which has similarities with our conclusions; see Remarks \ref{rmk:TW1} and \ref{rmk:TW2} for more details. Finally, we mention the work \cite{Ce05} by Cerrai, which shows stabilization by nonlinear multiplicative It\^o noise for another class of SPDEs, namely stochastic reaction-diffusion equations with multiplicative noise.

\textbf{Other approaches to no blow-up in infinite dimension:} Although here we focus on the superlinear noise case, there are other relevant approaches to regularization and no blow-up by noise, which we can roughly classify as follows: a) transport noise for non-smooth drifts, b) anomalous dissipation with transport-type, non-smooth-like noise, and c) noise multiplying the  ``driving term'' of the PDE under investigation. We mention a limited number of works for each class above, without claim of completeness. Concerning (a), Fedrizzi and Flandoli \cite{fedrizzi2013noise} proved that transport noise preserves the Sobolev regularity for the linear transport equations with a large class of non-smooth drifts. Concerning (b), Galeati \cite{Galeati2019OnTC} showed that, when taking a sequence of transport noises whose covariance matrices concentrate on the diagonal, anomalous dissipation emerges in the limit. Flandoli, Galeati, and Luo \cite{flandoli2021delayed, flandoli2021mixing} utilized this mechanism for several PDEs, including the 2D Euler equations, exploiting the extra dissipation. In particular, Flandoli and Luo \cite{flandoli2021high} demonstrated no blow-up with high probability for the 3D Navier-Stokes equations perturbed by a suitable transport noise.

Finally, concerning (c), we distinguish two cases. In one case, the noise multiplies the linear term of a dispersive PDE (modulated dispersive PDE): this approach has been used by Debussche and Tsutsumi in \cite{DebTsu2011}, Chouk and Gubinelli in \cite{ChoGub2015}, Chouk, Gubinelli, Li, Li and Oh in \cite{chouk2014nonlinear} to show well-posedness for nonlinear Schr\"odinger equation, KdV equation and other dispersive PDEs. In the other case, the noise multiplies the nonlinearity driving the PDE: this approach has been used by Gess and Souganidis in \cite{gess2014scalar} and Chouk and Gess in \cite{chouk2019path} for scalar conservation laws, as well as by Gassiat and Gess in \cite{gassiat2019} for Hamilton-Jacobi equations.

We also mention briefly another type of regularization, based on invariant measures, for (deterministic and stochastic) Euler equations: see e.g. \cite{AlbCru1990} and, more recently, \cite{FlaGroLuo2020}, \cite{Gro2020}. 
 
Finally, we highlight that there are situations where noise neither regularizes the system nor prevents blow-ups. For instance, in the finite-dimensional case, Scheutzow demonstrated a 2D example in \cite{doi101080Sch} where the solution becomes explosive when the dynamic system is perturbed by additive white noise, in which case the noise enhances the blow-up. In infinite-dimensional settings, de Bouard and Debussche showed in \cite{deBDeb2002} and in \cite{deBDeb2005} that the additive and multiplicative noise, white in time and correlated in space, can enhance blow-up, for example extending blow-up to every initial condition or accelerating blow-up. Hofmanov\'a, R. Zhu, and X. Zhu showed in \cite{hofmanova2019non}, using convex integration techniques, that stochastic 3D Navier-Stokes equations perturbed by additive, linear, or nonlinear noise do not have unique solution laws. Furthermore, Hofmanov\'a, Lange, and Pappalettera demonstrated in \cite{hofmanova2022global} that the 3D Euler equations, when perturbed by transport Stratonovich noise, have more than one probabilistically strong solution, making the equations ill-posed in H\"older spaces.

\textbf{Main result and idea of the proof:} Informally, our key results can be expressed in the following manner (see Theorems~\ref{thm:ApplicationToEuler}, \ref{thm:exist_global}, and~\ref{Unique} for precise statements). Concerning Euler equations, we show:

\begin{enumerate}[i)]
\item \textit{On a smooth, bounded two or three-dimensional domain $D$, the stochastic incompressible Euler equations~\eqref{EulerSto} with
\begin{align}\label{noise:superlinear}
\sigma(u) = c(1+\Vert u\Vert_{W^{1,\infty}(D)}^2)^{\beta/2} u
\end{align}
have a unique global-in-time strong solution in $H^s(D)$, provided that $s > d/2 + 2$, $u_0 \in H^{s+1}(D)$, and $\beta > 1/2, ~c > 0$ or $\beta = 1/2$ with a sufficiently large $c$.
}  
\end{enumerate}

The above result is a consequence of a no-blow-up result for general SPDEs, which we also show in this paper:
\begin{enumerate}[ii)]
\item \textit{Let $E_1\subset\subset E_0\hookrightarrow E_{-1}$ be dense embeddings of separable Hilbert spaces, with $E_1\subset\subset E_0$ being compact. Consider the SDE
\begin{align}\label{eq:SDE0}
dX_t =b(X_t)dt +\sigma(X_t)dW_t,
\end{align}
and assume that \emph{informally}, on a suitable dense set in $E_1$,
\begin{align*}
   &\langle b(x),x\rangle_{E_1} \le g(x)\|x\|_{E_1}^2,\\
   &\sigma(x)=f(x)x \text{ on }\{g\ge G\},\\
   &|f(x)|^2\ge 2g(x)-K \text{ on }\{g\ge G\},
\end{align*}
for suitable real-valued functions $f$ and $g$ and constants $G$ and $K$. Then, for every initial condition $X_0$ in $E_1$, the SDE \eqref{eq:SDE0} admits a weak, \emph{global in time} solution in $E_0$. Furthermore, if $b$ and $\sigma$ satisfy respectively a local monotonicity-type condition and a local Lipschitz-type condition, the solution also possesses pathwise uniqueness.}
\end{enumerate}

As mentioned above, our strategy to avoid blow-up for \eqref{eq:SDE0} is based on finding a suitable Lyapunov function and applying a compactness argument. Precisely, we take a Galerkin approximation $(X^n)_n$ of our SPDE and show morally that, with the choice of $\sigma$ as in the main result, $V(x)=\log \|x\|_{E_1}$ is a Lyapunov function for $X^n$, uniformly in $n$, thus getting a uniform, \emph{global in time} bound on the $E_1$ norm of $X^n$. The key computation to show the Lyapunov function property for $V(x)=\log \|x\|_{E_1}$ is morally the following one (we call $\mathcal{L}$ the generator of the SDE):
\begin{align*}
    \mathcal{L}\left(V\right)(x) = \frac {\langle b(x), x\rangle_{E_1}}{\lVert x \rVert_{E_1}^2} + \frac 12 \langle \sigma(x), D^2V(x) \sigma(x)\rangle
   \le g(x) - \frac12 |f(x)|^2 \le \frac{K}{2}.
\end{align*}
Then by a classical argument, using a stochastic Aubin-Lions lemma, we show that the family $(X^n)_n$ is tight in $C([0,T];E_0)$ and that any limit point converges weakly to a solution to the SPDE \eqref{eq:SDE0}. Along the way, we provide a general (though classical) lemma for the convergence of stochastic integrals, in the spirit of \cite{debussche2011local}, which might be of its own interest. Finally we show a strong uniqueness result and apply our general result to the case of Euler equations.

We remark that our focus here is on global well-posedness by superlinear noise, rather than on local well-posedness. Some assumptions could be weakened (see Remark \ref{rmk:Euler_optimality} for the Euler equations) by a finer analysis.

Motivated by the Wong-Zakai principle, one may ask if an analogous no blow-up result holds with Stratonovich integral in place of It\^o integral. The answer with a one-dimensional radial noise as here is negative: our noise fails to prevent blow-up when integrated in the Stratonovich sense, even in one dimension, see the SDE \eqref{eq:example_1D_Strat} and more in general Proposition \ref{prop:Strat_blowup}. However, several months after the conclusion of this paper, the first author demonstrated in \cite{Bagnara2023} that a no blow-up result, similar to the one presented here, can be obtained also for a Stratonovich perturbation with a higher-dimensional diffusion coefficient, inspired by the one introduced in \cite{maurelli2020non}; more precisely, in the Stratonovich case at least two independent noises are needed (see Remark \ref{rmk:Strat_positive_Bagnara}).

\textbf{Organization of the paper:} Our paper is organized as follows. In Section 2, we present our assumptions and main result in the general setting for equation \eqref{eq:SDE0}. Section 3 is dedicated to verifying that $\log\Vert x\Vert_{E_1}$ serves as a Lyapunov function for the approximate (Galerkin) SPDE models, which we use to establish the global existence and non-explosion of approximate solutions. We also demonstrate that the set of approximate solutions is precompact. In Section 4, we prove that any limit point of approximate solutions is a global weak solution to the original SPDE, and under certain Lipschitz-type conditions, we establish pathwise uniqueness for the weak solution. In Section 5, we apply the results of the general framework to the two and three-dimensional stochastic incompressible Euler equations, thereby concluding the global existence of a unique strong solution. Lastly, in Section 6, we consider briefly the case of Stratonovich noise in place of It\^o noise.

\section{Assumptions and Main Results}

In this section we describe the setting, the assumptions and our main result on an abstract PDE. We work mostly in the monotone and variational setting for SPDEs, see for example \cite{LiuRoc2015}, in particular concerning Gelfand triple and monotonicity-type assumptions; the main difference from \cite{LiuRoc2015} is that we do not have here a coercive linear operator.
Let $E_1$, $E_0$, and $E_{-1}$ be separable Hilbert spaces such that $E_1$ is compactly embedded in $E_0$ and $E_0$ is continuously embedded in $E_{-1}$, i.e.,
\begin{align*}
E_1\subset \subset E_0\hookrightarrow E_{-1},
\end{align*}
and let $(\Omega,\mathcal{A},(\mathcal{F}_t)_t,\mathbb P)$ be a filtered probability space satisfying the standard assumption. We address the stochastic differential equation
\begin{align}
dX_t = b(X_t)dt +\sigma(X_t)dW_t,\label{eq:SDE}
\end{align}
where $b,~\sigma$ are Borel functions from $E_0$ to $E_{-1}$ and $W$ is a one-dimensional Brownian motion on $(\Omega,\mathcal{A},(\mathcal{F}_t)_t,\mathbb P)$. For every $E_1$-valued random variable $X_0$ that is independent of $W$, we give the notions of weak solution and strong solution.

\begin{definition}\label{def:weak}
A weak solution to the equation \eqref{eq:SDE} is an object $(\Omega,\mathcal{A},(\mathcal{F}_t)_t,\Prob, (W_t)_t, (X_t)_t)$, such that $(\Omega,\mathcal{A},(\mathcal{F}_t)_t,\Prob)$ is a filtered probability space satisfying the standard assumption, $W$ is a one-dimensional Brownian motion on $(\Omega,\mathcal{A},(\mathcal{F}_t)_t,\Prob)$, and we have
\begin{itemize}
    \item $(X_t)_t$ is $(\mathcal{F}_t)_t$-progressively measurable (with values in $E_0$) and has continuous paths in $E_0$, $\Prob$-a.s.;
    \item $\Prob$-a.s., $t\mapsto b(X_t)$ is a locally integrable $E_{-1}$-valued function on $[0, \infty)$ and $t\mapsto \sigma(X_t)$ is a locally square-integrable $E_{-1}$-valued function on $[0, \infty)$;
    \item the following identity holds in $E_{-1}$  $\Prob$-a.s.:
\begin{align*}
X_t = X_0 + \int_0^t b(X_r) dr +\int_0^t \sigma(X_r) dW_r,\quad \forall t\ge0. 
\end{align*}
\end{itemize}
When $(\mathcal{F}_t)_t$ is the Brownian filtration (the smallest filtration satisfying the standard assumption and making $X_0$ $\mathcal{F}_0$-measurable and $W$ a Brownian motion), we say that the solution is strong.
\end{definition}

\begin{remark}
We assume $b,~\sigma$ to be $E_{-1}$ rather than $E_0$-valued to include fluid dynamic models whose drifts depend on the spatial gradient of $X_t$.
\end{remark}

In the remaining section, we summarize the assumptions needed for the main theorems. 
\begin{assumption}\label{Assump:Spaces}
 For separable Hilbert spaces $E_{-1},E_0,E_1$ introduced at the beginning, the compact embedding $E_1\subset \subset E_0$ and the continuous embedding $E_0\hookrightarrow E_{-1}$ are both dense. Moreover, there exist $\theta\in (0,1)$ and $M>0$ such that $\|v\|_{E_0}\le M\|v\|_{E_1}^{1-\theta}\|v\|_{E_{-1}}^\theta$ for all $v\in E_1$.
\end{assumption}

To employ an approximation argument, we exploit the Galerkin projections with respect to a vector basis specified by the following conditions.
\begin{assumption}\label{Assump:Projection}
There is a sequence of elements $\{e_n\}_{n\in\mathbb{N}^+}$ in $E_1$ whose span is dense in $E_{-1}$. Moreover, there is $C>0$ such that for every $x\in E_1$ and every $n\in\mathbb{N}^+$,
\begin{equation*}
\|\Pi_n x\|_{E_1}\le C\|x\|_{E_1},
\end{equation*}
where $\Pi_n$ is the orthogonal projector from $E_{-1}$ to $E^n$ and $E^n=\text{span}\{e_1,\ldots e_n\}$. 
\end{assumption}

\begin{remark}
Assumption \ref{Assump:Projection} is not strictly required. Indeed, using Assumption \ref{Assump:Spaces}, one can always proceed as in the proof of Lemma~\ref{Euler:basis} (see below) to construct a basis of $E_{-1}$ that is made of $E_{1}$ elements and is orthogonal in terms of both scalar products. (In this situation, we can choose $C=1$.) The reason why we impose Assumption \ref{Assump:Projection} is that the choice of $\{e_n\}_{n\in\mathbb{N}^+}$ also enters Assumptions \ref{Assump:drift} and \ref{Assump:sigma}.
\end{remark}

\begin{remark}
Combining Assumptions~\ref{Assump:Spaces} and \ref{Assump:Projection}, one can show that
\begin{equation*}
\|\Pi_n x-x\|_{E_{-1}}\rightarrow  0 \text{ as }n\to\infty,  \text{ for all }x\in E_{-1}.
\end{equation*}
\end{remark}
Now we introduce
\begin{equation} \label{eq:b_n&sigma_n}
b_n(x) = \Pi_n b(x),\quad \sigma_n(x) = \Pi_n \sigma(x),\quad x\in E^n.
\end{equation}

To state the assumptions on the drift~$b$, the initial condition~$X_0$, and the noise coefficient~$\sigma$, we denote the open ball centered at the origin with radius $R$ by $B_R$ and the complement by $B_R^c$. If there is a need to emphasize the underlying topology, we write $B_{R;\mathcal{H}}$ and $B^c_{R;\mathcal{H}}=\mathcal{H}\setminus B_{R;\mathcal{H}}$ (i.e., $B_{R;\mathcal{H}}$ is a subset of $\mathcal{H}$ and the radius of the ball is measured in the $\mathcal{H}$-norm, and $B^c_{R;\mathcal{H}}$ is the complement of $B_{R;\mathcal{H}}$ in $\mathcal{H}$). The closure of the ball is denoted by $\overline{B}_{R}$ or $\overline{B}_{R;\mathcal H}$.

We say that a function $f:\mathcal{H}\to \mathcal{K}$, with $\mathcal{H},~\mathcal{K}$ being separable Hilbert spaces, is bounded on balls if $\sup_{x\in \overline{B}_{R;\mathcal H}}\|f(x)\|_\mathcal{K}<\infty$ for every $R>0$. We use $a\vee b$ for $\max\{a,b\}$.

\begin{assumption}\label{Assump:drift} We assume that:
\begin{itemize}
\item[a)] $b:E_0\to E_{-1}$ is continuous and bounded on balls.
\item[b)] There exist $R\ge 1$ and a non-negative Borel function $g:E_1\to \mathbb{R}$, which is bounded on balls, such that 
\begin{equation*}
\langle b_n(x), x \rangle_{E_1} \le g(x)\lVert x\rVert_{E_1}^{2}\quad \text{ for every }x\in E^n\cap B_{R;E_1}^c \text{ and every } n\in \mathbb{N}.
\end{equation*}
\item[c)] The projected drift~$b_n:E^n\to E^n$ is locally Lipschitz for every $n\in \mathbb{N}$.
\item[d)] The initial condition $X_0$ belongs to $E_1$.
\item[e)] There exists an increasing function $k:\mathbb{R}\to \mathbb{R}_{+}$ such that
\begin{align*}
&\langle b(x)-b(y),x-y\rangle_{E_{-1}} \leq
k (\Vert x\Vert_{E_{0}}\vee\Vert y\Vert_{E_{0}} )\Vert x-y\Vert_{E_{-1}}^2\quad \text{for all } x,y\in E_{0}.
\end{align*}
\end{itemize}
\end{assumption}

Condition \eqref{Assump:drift}-(e) is known as local monotonicity condition and has been introduced in \cite{LiuRoc2010}.

\begin{assumption} \label{Assump:sigma} We assume that:
\begin{itemize}
\item[a)] $\sigma:E_0\to E_{-1}$ is continuous and bounded on balls.
\item[b)] There exist a Borel function $f:E_1\to \mathbb{R}$ which is bounded on balls and positive constants $G,~L$ such that
\begin{equation*}
\begin{cases}
    \sigma(x)=f(x)x\quad 
    &\text{for every }\,x\in B^c_{R;E_1} \cap \{g>G \},\\
    \lVert \sigma(x) \rVert_{E_1} \le L \lVert x \rVert_{E_1} \quad 
    &\text{for every }\,x\in B^c_{R;E_1} \cap \{g\le G \},
\end{cases}
\end{equation*}
where $g,R$ are the function and constant introduced in Assumption~\ref{Assump:drift}-(b).
\item[c)] There exists a non-negative constant $K$ such that
\begin{equation*}
|f(x)|^2 \geq 2g(x) -K \quad \text{for every }\,x\in B^c_{R;E_1} \cap \{g>G \}.
\end{equation*}
\item[d)] The projected noise coefficient~ $\sigma_n:E^n\to E^n$ is locally Lipschitz for every $n\in\mathbb{N}$.
\item[e)] There exists an increasing function $h:\mathbb{R}\to \mathbb{R}_+$ such that
\begin{align*}
&
\Vert \sigma(x)-\sigma(y)\Vert_{E_{-1}} \leq
h (\Vert x\Vert_{E_{0}}\vee\Vert y\Vert_{E_{0}} )\Vert x-y\Vert_{E_{-1}}\quad \text{for all }\,x,y\in E_{0}.
\end{align*}
\end{itemize}
\end{assumption}

 Below is our first main result of this paper, which follows immediately from Theorems~\ref{thm:exist_global} and~\ref{Unique} in Section~4.

\begin{theorem}\label{thm:exist_strong}
Under Assumptions \ref{Assump:Spaces}, \ref{Assump:Projection}, \ref{Assump:drift} and \ref{Assump:sigma}, there exists a unique global-in-time, strong solution to~\eqref{eq:SDE} with the initial condition $X_0$.
\end{theorem}

\begin{remark}
    Actually, as we will see in Theorem \ref{thm:exist_global}, the monotonicity-type Assumptions \ref{Assump:drift}-(e) and \ref{Assump:sigma}-(e) are only needed for uniqueness and not for global-in-time weak existence.
\end{remark}

\begin{remark}
    We emphasize that the main focus of the current paper is showing that a superlinear noise can provide global existence for some SPDEs and especially for 3D Euler equations. In particular, we do not focus our attention in optimizing certain assumptions, as for instance the regularity on the initial condition, see also Remark \ref{rmk:Euler_optimality}.
\end{remark}

\section{Precompactness of Approximate Solutions}

In this section, we construct global approximate solutions. The main trick in our argument is the use of $\log\|x\|_{E_1}$ as a Lyapunov function. We show that the approximate solutions are uniformly bounded in $E_1$ with a large probability. With this bound, we can then establish the tightness for the set of approximate solutions' probability distributions.

Consider the following finite-dimensional Galerkin approximations on $E^n$ to the SDE~\eqref{eq:SDE}, 
\begin{equation} \label{SDE_Galerkin}
    \begin{cases}
    dX^n_t=b_n(X^n_t)dt+\sigma_n(X^n_t)dW_t,\\
    X^n_0=\Pi_n X_0.
    \end{cases}
\end{equation} 
where $b_n,\sigma_n$ are defined in \eqref{eq:b_n&sigma_n} and $X_0\in E_1$. Since $b_n, \sigma_n$ are locally Lipschitz (by Assumptions \ref{Assump:drift}-(c) and \ref{Assump:sigma}-(d)), the local existence of a unique solution $X^n_t$ to~\eqref{SDE_Galerkin} is a classical result.

Let $R$ be the radius introduced in Assumption~\ref{Assump:drift}-(b), and let $a\in (0,\log(2R))$. We introduce a radially symmetric $C^2$ Lyapunov function $V:E_1 \to \R$ such that
\begin{align*}
\begin{cases}
    V(x)\ge a, &\forall x\in E_1, \\
    V(x)= a, &\forall x\in B_{R;E_1}, \\
    V(x)=\log\lVert x \rVert_{E_1}, \quad &\forall x\in B^c_{2R;E_1}.
\end{cases}
\end{align*}
We also ask $V$ to be non-decreasing along any ray from the origin.
Explicit computations show that for all $x\in B^c_{2R;E_1}$,
\begin{align*}
    \nabla V(x)&=\frac x{\lVert x \rVert_{E_1}^2},\\
    D^2V(x)&=\frac 1{\lVert x \rVert_{E_1}^2}\left(\Id -2\frac{x\otimes x}{\lVert x \rVert_{E_1}^2} \right).
\end{align*}

Now, we state a key lemma of this section. It proves the global-in-time existence of $X^n$ and provides a uniform Lyapunov-type bound for $X^n$ in probability, with respect to the $C([0,T]; E_1)$-topology. Here, $T$ is an arbitrarily prescribed positive number.

\begin{lemma} \label{Lemma:BoundednessInProb}
The unique local solution $X^n$ to \eqref{SDE_Galerkin} exists globally on $[0,T]$. Moreover, for every $\epsilon>0$, there exists $C_\epsilon>0$ such that
\begin{equation}\label{eq:sup_bounded_in_prob}
    \sup_n \Prob\left(\sup_{t\in[0,T]} \lVert X^n_t \rVert_{E_1} \ge C_\epsilon\right)\le \epsilon.
\end{equation}
\end{lemma}

\begin{proof}[Proof of Lemma~\ref{Lemma:BoundednessInProb}]
For ease of applying It\^o's lemma, we endow the $E_n$-space with the $E_1$-scalar product for every $n\in \mathbb{N}$. By the It\^o formula, $V(X^n)$ has the following stochastic differential up to the maximal existence time of the process $X^n$,
\begin{align}\label{Ito-LV}
    dV(X^n_t)= \mathcal{L}_n V(X^n_t) dt +\langle\nabla V(X^n_t), \sigma_n(X^n_t)\rangle dW_t.
\end{align}
Above, $\mathcal{L}_n$ denotes the generator of this SDE. To be specific,
\begin{align*}
    \mathcal{L}_n\left(V\right)(x)&=\langle \nabla V(x), b_n(x) \rangle + \frac 12 \langle \sigma_n(x), D^2V(x) \sigma_n(x)\rangle.
\end{align*}

We now show that $\mathcal{L}_n(V)$ is globally bounded in $E^n$, uniformly in $n\in\mathbb{N}$. First, note
that $V$ is constant in $B_{R;E_1}$ and so $\mathcal{L}_n(V)(x)= 0$ for all $x\in B_{R;E_1}$. Next, we consider $x$ in $E^n$ such that $R\le \lVert x \rVert _{E_1} \le 2R$. By the radial symmetry and the non-decreasing property of $V$, $\nabla V(x)=w(\|x\|_{E_1})x/\lVert x \rVert_{E_1}$ for some $C^1$ function $w:\R_+\to \R_+$. 
Therefore, for $x \in E^n$ with $R\le \lVert x \rVert _{E_1} \le 2R$, we have
\begin{align*}
    &\mathcal{L}_n\left(V\right)(x)=\frac{w(\|x\|_{E_1})} {\lVert x\rVert_{E_1}}\langle b_n(x),x \rangle + \frac 12 \langle \sigma_n(x), D^2V(x) \sigma_n(x)\rangle 
    \\& \quad
    \le w(\|x\|_{E_1})g(x)\lVert x\rVert_{E_1} + 
    \frac 12
    \begin{cases}
         \lVert D^2V(x) \rVert\, |f(x)|^2 \lVert x \rVert^2_{E_1} &\qquad \text{if }x\in \{g(x)>G \}\\
        \lVert D^2V(x) \rVert\, \lVert \sigma_n(x) \rVert^2_{E_1} &\qquad \text{if }x\in \{g(x)\le G \}
    \end{cases}
    \\ & \quad
    \le w(\|x\|_{E_1})g(x)\lVert x\rVert_{E_1} + \frac 12 \lVert D^2V(x) \rVert\, \lVert x \rVert^2_{E_1} \max\left( |f(x)|^2, L^2\right).
\end{align*}
Due to the continuity of $w,~D^2V$ and the boundedness of $f,~g$ on $E_1$-balls, the right side of the above inequality is uniformly bounded in $n$. 
Lastly, we consider $x\in B^c_{2R;E_1}\cap E^n$. In this case, we have
\begin{align*}
   \mathcal{L}_n\left(V\right)(x)& = \frac {\langle b_n(x), x\rangle_{E_1}}{\lVert x \rVert_{E_1}^2} + \frac 12 \langle \sigma_n(x), D^2V(x) \sigma_n(x)\rangle
   \le g(x) + \frac 12 \langle \sigma_n(x), D^2V(x) \sigma_n(x)\rangle.  
\end{align*}
If $g(x)>G$, then by Assumption~\ref{Assump:sigma}-(c) we get
\begin{align*}
   \mathcal{L}_n\left(V\right)(x) \le 
        g(x)-\frac12|f(x)|^2  
    \le  \frac K2.  
\end{align*}
If $g(x) \leq G$, then by Assumption~\ref{Assump:sigma}-(b) (and Assumption~\ref{Assump:Projection}) instead we get
\begin{align*}
   \mathcal{L}_n\left(V\right)(x) \le 
        G +\frac12 \lVert D^2V(x) \rVert\, \lVert \sigma_n(x) \rVert^2_{E_1} 
    \le  G+\frac12 C^2 L^2.  
\end{align*}
Combining all cases, we conclude that $\mathcal{L}_n(V)$ is bounded in $E_n$ uniformly in $n$. Since $0<a\leq V$, there must be a positive constant $c$ such that 
\begin{equation}
\mathcal{L}_n\left(V\right)\le c V.\label{eq:Lyap_bd}
\end{equation}
This bound implies the global existence of the approximate solutions $X^n$ for every $n\in \mathbb{N}$ (see~\cite[Theorem 3.5]{khasminskii2011stochastic}).

To prove \eqref{eq:sup_bounded_in_prob}, we employ a similar argument as in \cite[Theorem 3.5]{khasminskii2011stochastic}. Let $M\gg a$, and let $\tau^n_M$ be the stopping time when $V(X^n_t)$ first hits $M$.
We integrate both sides of \eqref{Ito-LV} over the interval $[0, t \land \tau_M]$, take the expected value, and apply \eqref{eq:Lyap_bd}. We obtain
\begin{align*}    &\E[V(X^n_{t\land\tau^n_M})] =\E[V(X^n_0)]
    +\E\int_0^{t\land \tau^n_M}\mathcal{L}_n(V)(X^n_s)\,ds
    \\ & \qquad
    \le \E[V(X^n_0)]+\E\int_0^{t\land \tau^n_M}cV(X^n_s)\,ds, \quad t\in(0,T).
\end{align*}
Applying Gronwall's lemma and Assumption \ref{Assump:Projection}, we arrive at
\begin{equation*}
    \E[V(X^n_{T\land\tau^n_M})]\le\E[V(X^n_0)]e^{cT}\le \E[V(C x_0)]e^{cT},
\end{equation*}
which by the Markov inequality implies
\begin{equation} 
    \Prob\left(\sup_{t\in[0,T]} V(X^n_t) \ge M\right)=\Prob\left(V(X^n_{T\wedge \tau^n_M})\ge M\right)
    \le\frac{\E[V(C x_0)] e^{cT}}M.\label{eq:Vp_Markov}
\end{equation}
Then, for any $\epsilon>0$, we obtain \eqref{eq:sup_bounded_in_prob} choosing a sufficiently large value for $M$.
\end{proof}

The next result provides a uniform bound in probability for $X^n$ in $C^{0,\alpha}([0,T];E_{-1})$ with $\alpha\in (0,1/2)$. Here, $C^{0,\alpha}([0,T];E_{-1})$ refers to the H\"{o}lder spaces defined by the norm
\begin{align*}
&\|f\|_{C^{0,\alpha}(E_{-1})} := \|f\|_{C(E_{-1})} + [f]_{C^{0,\alpha}(E_{-1})} := \sup_{t\in[0,T]}\|f(t)\|_{E_{-1}} 
+
\sup_{t,s\in[0,T]; t\neq s}\frac{\lVert f(t) -f(s)\rVert_{E_{-1}}}{|t-s|^{\alpha}}.
\end{align*}

\begin{lemma} \label{Lemma:HolderBoundednessInProb}
Let $\alpha\in (0,1/2)$. For every $\epsilon>0$, there exists a positive constant $C_\epsilon$ such that
\begin{equation}\label{eq:Holder_bounded_in_prob}
\sup_n\Prob\left( \lVert X^n \rVert_{C^{0,\alpha}(E_{-1})} \ge C_\epsilon\right)\le \epsilon. 
\end{equation}
\end{lemma}

    For the proof, we recall the Sobolev embedding $W^{\beta,p}\hookrightarrow C^{0,\alpha}$, where $1<p<\infty$ and $\alpha\le \beta-1/p$. More precisely, we may apply e.g.~\cite[Theorem B.1.5]{da1996ergodicity}, getting
\begin{equation*}
    \|f(t)-f(s)\|_{E_{-1}}^p \le C|t-s|^{\beta p -1} \int_0^T\int_0^T \frac{\|f_t-f_s\|_{E_{-1}}^p}{|t-s|^{1+\beta p}}\,ds\,dt
\end{equation*}
    and conclude that
\begin{equation}
    [f]_{C^{0,\alpha}(E_{-1})}^p \le C[f]_{W^{\beta,p}(E_{-1})}^p := C\int_0^T\int_0^T \frac{\|f_t-f_s\|_{E_{-1}}^p}{|t-s|^{1+\beta p}}\,ds\,dt \label{eq:Sobolev_emb_time}
\end{equation}
for a continuous function $f:[0,T]\to E_{-1}$, $1<p<\infty$, and $\alpha\le\beta-1/p$.

\begin{proof}[Proof of Lemma~\ref{Lemma:HolderBoundednessInProb}]
Take $\epsilon>0$. By \eqref{eq:sup_bounded_in_prob} and the continuous embedding from $E_1$ to $E_{-1}$, there exists $c_{\epsilon}>0$ such that
\begin{equation*}
\sup_n \Prob\left( \lVert X^n \rVert_{C(E_{-1})} \ge c_\epsilon\right)\le \epsilon.
\end{equation*}
Then, it suffices to establish a uniform bound in probability for the semi-norm $[X^n]_{C^{0,\alpha}(E_{-1})}$.

For $M>0$, we introduce the stopping time
\begin{equation*}
    \hat{\tau}^n_M = \inf\big\{t\ge 0\mid \|X^n_t\|_{E_1}\ge M\big\},
\end{equation*}
 which is the first exit time of $X^n$ from the open ball $B_{M;E_1}$, and the event
\begin{equation*}
    A^n_M:=\left\{\sup_{t\in[0,T]}\left\Vert X^n_t\right\Vert_{E_1} \ge M\right\}.
\end{equation*}
Clearly, $A^n_M\supseteq \{\hat{\tau}^n_M< T\}$. Moreover, $\sup_n \Prob (A^n_M) < \epsilon$ provided that $M$ is a sufficiently large number (see~\eqref{eq:sup_bounded_in_prob}).
For such an $M$, we have
\begin{align*}
&\Prob\left( [X^n]_{C^{0,\alpha}(E_{-1})} > a\right)
\le \Prob\left(~(A_M^n)^c,~[X^n]_{C^{0,\alpha}(E_{-1})}> a\right)+\Prob\left( A^n_M\right)\\
& \qquad\le \Prob\left(\left[X^n_{\cdot\wedge \hat{\tau}^n_M}\right]_{C^{0,\alpha}(E_{-1})}> a \right)+\epsilon
\le \frac{1}{a^p}\E\left[\left[X^n_{\cdot\wedge \hat{\tau}^n_M}\right]_{C^{0,\alpha}(E_{-1})}^p \right]+\epsilon,
\end{align*}
 whenever $a>0$ and $p>0$. It remains to show that
\begin{equation}
\sup_n\E\left[\left[X^n_{\cdot\wedge \hat{\tau}^n_M}\right]_{C^{0,\alpha}(E_{-1})}^p \right]<\infty,\label{eq:C_alpha_exp}
\end{equation}
as \eqref{eq:C_alpha_exp} implies that
$\sup_n\Prob\left( [X^n]_{C^{0,\alpha}(E_{-1})} > a_\epsilon\right)<2\epsilon
$ for some $a_\epsilon>0$, concluding the proof.

For this purpose, we use the equations~\eqref{SDE_Galerkin} and obtain
\begin{equation*}
    \left\Vert X^n_{t\wedge \hat\tau^n_M} -X^n_{s\wedge \hat\tau^n_M}\right\Vert^p_{E_{-1}}\le C_p \left\Vert\int_s^t b_n(X^n_{r\wedge \hat\tau^n_M})\,dr \right\Vert^p_{E_{-1}} + C_p \left\lVert\int_s^t\sigma_n(X^n_{r\wedge \hat\tau^n_M})\,dW_r \right\rVert^p_{E_{-1}},
\end{equation*}
where $s,t \in [0,T]$. Recall that $b_n=\Pi_n b$. Hence, the drift term can be bounded as follows,
\begin{align*}
& \left\Vert\int_s^t b_n(X^n_{r\wedge \hat\tau^n_M})\,dr \right\Vert^p_{E_{-1}}
    \le |t-s|^{p} \sup_{r\in[0,T]} \left\Vert b_n(X^n_{r\wedge \hat\tau^n_M}) \right\Vert^p_{E_{-1}}
    \le |t-s|^{p} \sup_{r\in[0,T]} \left\Vert b(X^n_{r\wedge \hat\tau^n_M}) \right\Vert^p_{E_{-1}}.
\end{align*}
If $p>1$, the stochastic integral can be addressed using the Burkholder-Davis-Gundy inequality,
\begin{align*} 
&\E\left[\left\lVert\int_s^t\sigma_n(X^n_{r\wedge \hat\tau^n_M})\,dW_r \right\rVert^p_{E_{-1}} \right]
    \le C_p\, \E\left[\left( \int_s^t  \left\lVert \sigma_n(X^n_{r\wedge \hat\tau^n_M})\right\rVert^2_{E_{-1}}\, dr \right)^\frac p2\right]
    \\
    &\quad
    \le C_p\, |t-s|^{\frac p2}\E\left[ \sup_{r\in[0,T]} \left\lVert \sigma_n(X^n_{r\wedge \hat\tau^n_M})\right\rVert^p_{E_{-1}} \right]
    \le C_p\, |t-s|^{\frac p2}\E\left[ \sup_{r\in[0,T]} \left\lVert \sigma(X^n_{r\wedge \hat\tau^n_M})\right\rVert^p_{E_{-1}} \right].
\end{align*}
Due to the continuous embedding of $E_1$ into $E_0$ and the boundedness on balls of $b$ and $\sigma$ from $E_0$ into $E_{-1}$ (by Assumptions \ref{Assump:drift}-(a) and \ref{Assump:sigma}-(a)), there exist positive constants $M_b,M_\sigma$ depending on $M$ such that 
\begin{align*}
    \sup_{r\in[0,T]} \left\lVert b(X^n_{r\wedge \hat\tau^n_M})\right\rVert^p_{E_{-1}}
    \le  M_b,
    \qquad
    \sup_{r\in[0,T]} \left\lVert \sigma(X^n_{r\wedge \hat\tau^n_M})\right\rVert^p_{E_{-1}}
    \le  M_\sigma
\end{align*}
for every $n\in\mathbb{N}$. Therefore, 
\begin{equation*}
\E\left[\left\Vert X^n_{t\wedge \hat\tau^n_M} -X^n_{s\wedge \hat\tau^n_M}\right\Vert^p_{E_{-1}}\right] \le C_{p,T,M}|t-s|^{p/2},
\end{equation*}
and therefore,
\begin{align*}
   &\E\left[ [X^n_{\cdot\wedge \hat{\tau}^n_M}]_{W^{\beta,p}(E_{-1})}^p \right]
   = \int_0^T\int_0^T\frac{\E \lVert X^n_{t\wedge \hat\tau^n_M} -X^n_{s\wedge \hat\tau^n_M}\rVert^p_{E_{-1}}}{|t-s|^{1+\beta p}}\,ds\,dt
   \\ & \qquad\qquad
    \le C_{p,T,M}\int_0^T\int_0^T|t-s|^{(1/2-\beta)p-1}\,ds\,dt.
\end{align*}
The right side of the above inequality is uniformly bounded with respect to $n$ when $\beta<1/2$. Given $\alpha\in (0,1/2)$, we can always choose $\beta, p$ so that $\beta<1/2$, $1<p<\infty$, and $\alpha<\beta-1/p$, then the inequality above and the Sobolev embedding \eqref{eq:Sobolev_emb_time} imply \eqref{eq:C_alpha_exp}, completing the proof of the lemma.
\end{proof}

Recall the Aubin-Lions-Simon theorem, which claims that
$L^\infty([0,T];E_1)\cap W^{\alpha, \infty}([0,T];E_{-1})$ is compactly embedded in $C([0,T];E_0)$ under the Assumption~\ref{Assump:Spaces} when $\alpha\in (0,1)$ (see~\cite[Corollary 9]{Si86}). Idenfitying functions in $W^{\alpha, \infty}([0,T];E_{-1})$ with their continuous modifications, we have
\[
L^\infty([0,T];E_1)\cap C^{0, \alpha}([0,T];E_{-1})
\subset\subset C([0,T];E_0).
\] 
Then, the family of the laws of $X^n$ is tight in $C([0,T];E_0)$ by Lemmas \ref{Lemma:BoundednessInProb} and \ref{Lemma:HolderBoundednessInProb}, and then the set of coupled laws of $\{(X^n,W )\}_{n\in\mathbb{N}}$ is tight in the product Polish space $C\left([0,T];E_0\right)\times C\left([0,T];\R\right)$, as a consequence of the tightness of the two components. Applying Prokhorov's theorem, we conclude that $\{(X^n,W )\}_{n\in\mathbb{N}}$ is a precompact set in the topology of weak convergence.

\section{Global Solution and Pathwise Uniqueness}\label{sec4}

In this section, we identify a weak cluster point of $\{(X^n,W )\}_{n\in\mathbb{N}}$ as a martingale solution to \eqref{eq:SDE}. Let us take a weakly convergent subsequence and still denote it by $\{(X^n,W )\}_{n\in\mathbb{N}}$ for brevity. By the Skorokhod theorem (see e.g.~\cite[Theorem 2.4]{da2014stochastic}), there exist a probability space $(\widetilde\Omega,\widetilde{\mathcal{A}},\widetilde{\mathbb P})$ and $C\left([0,T];E_0\right)\times C\left([0,T];\R\right)$-valued random variables $\{(\tilde X^n,\tilde W^n)\}_{n\in\N}$, $(\tilde X, \tilde W)$ living in this probability space, such that $(\tilde X^n,\tilde W^n)$ is distributed the same as $( X^n,W^n)$ for every $n\in\mathbb{N}$, and the sequence $(\tilde X^n,\tilde W^n)$ converges to $(\tilde X, \tilde W)$ in $C\left([0,T];E_0\right)\times C\left([0,T];\R\right)$ $\widetilde \Prob$-a.s.. 

For technicalities, we denote the filtration generated by $\tilde X,\tilde W$, and $\widetilde \Prob$-null sets by $(\mathcal{\widetilde G}_t)_{t\in[0,T]}$, and we define $\mathcal{\widetilde F}_t \coloneqq \cap_{s>t} \mathcal{\widetilde G}_s$, a filtration satisfying the standard assumption. The filtrations $(\mathcal{\widetilde G}^n_t)_{t\in[0,T]}$ and $(\mathcal{\widetilde F}^n_t)_{t\in[0,T]}$ for $\tilde X^n,\tilde W^n$ are defined similarly. With regard to $\tilde W$, $\tilde W^n$, and $\{(\tilde X^n,\tilde W^n)\}_{n\in\N}$, we have the next two (standard and technical) lemmas,\ whose proofs are postponed to the Appendix. 

\begin{lemma}\label{BM}
For every $n\in\mathbb{N}$, $\tilde W^n$ is an $(\mathcal{\widetilde F}^n_t)_{t\in[0,T]}$-adapted Brownian motion, and the limit process $\tilde W$ is an $(\mathcal{\widetilde F}_t)_{t\in[0,T]}$-adapted Brownian motion.
\end{lemma}

\begin{lemma}\label{WS}
For every $n\in \N$, $(\widetilde \Omega,  \widetilde{\mathcal{A}},(\widetilde{\mathcal{F}}^n_t)_{t}, \widetilde \Prob,\tilde X^n, \tilde W^n)$ is a weak solution of \eqref{SDE_Galerkin}. 
\end{lemma}

We aim to show that $(\tilde X,\tilde W)$, the $\Prob$-a.s.\ limit of $\{(\tilde X^n,\tilde W^n)\}_{n\in\N}$, is a global weak solution of the limit model~\eqref{eq:SDE}. This requires us to study the convergence for the stochastic integrals. The literature offers several classical arguments on this topic, beginning with the foundational work of Skorokhod \cite[Theorem on p. 32]{Skorokhod1965} and followed by numerous variants (see, for instance, \cite[Lemma 5.2]{GyongyMartinez2001}, \cite[Lemma 3.2]{Luo2011} and \cite[Section 4.3.5]{bensoussan1995}). Without claiming novelty, we present here an infinite-dimensional convergence lemma inspired by \cite[Lemma 2.1]{debussche2011local} and useful for future reference. Compared to \cite[Lemma 2.1]{debussche2011local}, we refine their proof to establish the convergence of stochastic integrals in a stronger topology, specifically convergence in probability with respect to the supremum norm in time, rather than the $L^2$-in-time norm. Given the classical nature of this result, we postpone its proof to the appendix.

Let $\mathfrak{U}, \mathcal{H}$ be two separable Hilbert spaces and $\{\mathbf{e}_k\}_{k\in \N}$ be a complete orthonormal basis of $\mathfrak{U}$. We use $\mathcal{L}(\mathfrak{U}, \mathcal{H})$ for the set of bounded linear operators from $\mathfrak{U}$ to $\mathcal{H}$ and denote by 
\begin{equation*}
    L_2(\mathfrak{U}, \mathcal H)\coloneqq \bigg\{ G \in \mathcal{L}(\mathfrak{U}, \mathcal{H})\, : \, \sum_{k\ge 1} \lVert G \mathbf{e}_k \rVert_{\mathcal H}^2 <  \infty\bigg\}
\end{equation*}
the space of Hilbert-Schmidt operators from  $\mathfrak{U}$ to $ \mathcal{H}$. We endow $L_2(\mathfrak{U}, \mathcal H)$ with its Borel $\sigma$-algebra and the elements in $L_2(\mathfrak{U}, \mathcal H)$ with the norm 
\begin{equation*}
    \Vert G\Vert_{L_2(\mathfrak{U}, \mathcal H)}= \bigg( \sum_{k\ge 1} \lVert G \mathbf{e}_k \rVert_{\mathcal H}^2 \bigg)^{1/2}, \quad \forall G \in L_2(\mathfrak{U}, \mathcal H).
\end{equation*}
When there is no risk of confusion, we abbreviate $G_k=G\mathbf{e}_k$.
Besides, let $\{W_k\}_{k\in\mathbb{N}}$, $\{W_k^n\}_{k\in\mathbb{N}}$ be sequences of independent Brownian motions on a probability space $(\Omega, \mathcal{ A}, \mathbb{ P})$. Then, $ W= \sum_{k\ge 1} W_k \mathbf{e}_k$ defines a cylindrical Wiener process over $\mathfrak{U}$, and $ W^n= \sum_{k\ge 1} W^n_k \mathbf{e}_k$ are cylindrical Wiener processes over $\mathfrak{U}$. Note that $W$ and $W^n$ are adapted to the augmentation of their respective natural filtrations, say $(\mathcal{ F}_t)_{t\in[0,T]}$ and $(\mathcal{ F}^n_t)_{t\in[0,T]}$. Within this context, we claim the following convergence.

\begin{lemma} \label{lemma:StochIntConvergence}
Suppose that $W= \sum_{k\ge 1} W_k \mathbf{e}_k$ and $ W^n= \sum_{k\ge 1} W^n_k \mathbf{e}_k$ are cylindrical Wiener processes over $\mathfrak{U}$, adapted with respect to the filtrations $(\mathcal{ F}_t)_{t\in[0,T]}$ and $(\mathcal{ F}^n_t)_{t\in[0,T]}$ respectively. Also suppose that $\Prob$-a.s., $G$ is an $(\mathcal{ F}_t)_{t\in[0,T]}$-progressively measurable process in $L^2([0,T]; L_2(\mathfrak{U}, \mathcal H))$ and $\{G^n\}_{n\in\mathbb{N}}$ are $(\mathcal{ F}^n_t)_{t\in[0,T]}$-progressively measurable processes in $L^2([0,T]; L_2(\mathfrak{U}, \mathcal H))$ for all $n\in\mathbb{N}$. As $n\to\infty$, if
\begin{subequations}\label{eq:ConvAsmp}
\begin{gather}
   W^{n}_k \rightarrow  W_k  \quad
  \textrm{ in probability in }
  C([0,T]; \R)
  \textrm{ for all }
   k\in\mathbb{N},
  \label{eq:ConvAsmpNoise}\\
   G^{n} \rightarrow  G \quad
  \textrm{ in probability in }
  L^2([0,T]; L_2(\mathfrak{U},  \mathcal H)), \label{eq:ConvAsmpStochIntegrand}
\end{gather}
\end{subequations}
then
\begin{equation}\label{eq:StochasticVitaliConc}
\sup_{t\in[0,T]}\left\Vert \int_0^t  G^n\, d W^n - \int_0^t  G\, d W \right\Vert_{\mathcal H} \to 0
\quad \textrm{ in probability}.
\end{equation}
\end{lemma}

\begin{theorem}[Global Existence]\label{thm:exist_global}
Under the Assumptions \ref{Assump:Spaces}, \ref{Assump:Projection}, \ref{Assump:drift}-(a,b,c,d), and \ref{Assump:sigma}-(a,b,c,d), there exists a global-in-time, weak solution to~\eqref{eq:SDE} with initial condition distributed as $X_0$.
\end{theorem}
\begin{proof}[Proof of Theorem~\ref{thm:exist_global}]
We have proved that on a common probability space $(\widetilde{\Omega}, \widetilde{\mathcal{A}}, \widetilde{\Prob})$ there exist random variables $\{(\tilde X^n,\tilde W^n)\}_{n\in\N}$ and $(\tilde X, \tilde W)$ such that, $\widetilde{\Prob}$-a.s., $(\tilde X^n,\tilde W^n)$ converges to $(\tilde X, \tilde W)$ in $C\left([0,T];E_0\right)\times C\left([0,T];\R\right)$ and satisfies
\begin{equation}
        \tilde X^n_t=\tilde X^n_0 +\int_0^t b_n(\tilde X^n_s)\,ds+\int_0^t\sigma_n(\tilde X^n_s)\,d \tilde W^n_s,\quad \forall t\in[0,T].\label{eq:tilde_X}
\end{equation}
By the continuous embedding from $E_0$ to $E_{-1}$, $\Vert \tilde X^n_t-\tilde X_t\Vert_{E_{-1}}\to 0$ at all $t$ in $[0,T]$ $\widetilde{\Prob}$-a.s.\ as $n\to\infty$. Let us analyze the integral terms in~\eqref{eq:tilde_X}. Concerning the deterministic integral, we have
\begin{align*}
    &\left\Vert
    \int_0^t b_n(\tilde X^n_s)\,ds
    -\int_0^t b(\tilde X_s)\,ds
    \right\Vert_{E_{-1}}
  \\
  &\qquad \le 
    \int_0^t \left\Vert \Pi_n b(\tilde X^n_s)-\Pi_n b(\tilde X_s) \right\Vert_{E_{-1}} ds
    + 
    \int_0^t \left\Vert \Pi_n b(\tilde X_s)-b(\tilde X_s) \right\Vert_{E_{-1}} ds
    \\
    &\qquad \le
    \int_0^t \left\Vert b(\tilde X^n_s)-b(\tilde X_s) \right\Vert_{E_{-1}} ds
    + 
    \int_0^t \left\Vert \Pi_n b(\tilde X_s)-b(\tilde X_s) \right\Vert_{E_{-1}} ds.
\end{align*}
Using the continuity and boundedness on balls of $b$ (see Assumption~\ref{Assump:drift}), we may apply the dominated convergence theorem and conclude that the last two integrals in the above inequality converge to zero $\widetilde{\Prob}$-a.s., then as $n\to\infty$, 
\begin{equation*}
     \int_0^t b_n(\tilde X^n_s)\,ds
    \to \int_0^t b(\tilde X_s)\,ds \quad \text{in } E_{-1}\quad\widetilde{\Prob}\text{-a.s.}
\end{equation*}

For the stochastic integral term in \eqref{eq:tilde_X}, we apply Lemma~\ref{lemma:StochIntConvergence}. Then, it is enough to show that 
\begin{equation*}
    \int_0^t \left\Vert \sigma_n(\tilde X^n_s)-\sigma(\tilde X_s)\right\Vert^2_{E_{-1}}\,ds \to 0 \quad \widetilde{\Prob}\text{-a.s.}
\end{equation*}
as $n\to \infty$, which can be proved similarly as for the deterministic integral, via the following splitting,
\begin{equation*}
\left\Vert\sigma_n(\tilde X^n_s)-\sigma(\tilde X_s)\right\Vert_{E_{-1}}^2 \le 2\left\Vert\Pi_n \sigma (\tilde X^n_s)-\Pi_n \sigma(\tilde X_s)\right\Vert_{E_{-1}}^2 +2\left\Vert \Pi_n \sigma(\tilde X_s)-\sigma(\tilde X_s) \right\Vert_{E_{-1}}^2
\end{equation*}
and using the continuity and boundedness on balls of $\sigma$ (see Assumption~\ref{Assump:sigma}).
Therefore, by passing $n$ to infinity along a subsequence, we have that
\begin{equation*}
     \int_0^t\sigma_n(\tilde X^n_s)\,d \tilde W^n_s
    \to \int_0^t\sigma(\tilde X_s)\,d \tilde W_s\quad \text{in } E_{-1}\quad\widetilde{\Prob}\text{-a.s.}
\end{equation*}
Gathering all terms and passing to the $\widetilde{\Prob}$-a.s.\ limit in $E_{-1}$ in the equation \eqref{eq:tilde_X}, we obtain that, for every $t$,
\begin{equation*}
    \tilde X_t=\tilde X_0 +\int_0^t b(\tilde X_s)\,ds+\int_0^t\sigma(\tilde X_s)\,d \tilde W_s\quad \widetilde{\Prob}\text{-a.s.}
\end{equation*}
By the continuity of $\tilde X$-paths, the above equality holds for all $t\in [0,T]$ in a $\widetilde{\Prob}$-full measure set that does not depend on $t$. Hence, $(\widetilde{\Omega}, \widetilde{\mathcal{A}}, (\widetilde{\mathcal{F}}_t)_{t\in[0,T]}, \widetilde{\Prob},\tilde X, \tilde W)$ is a weak solution to \eqref{eq:SDE}, ending the proof.
\end{proof}

\begin{remark}\label{rmk:TW1}
    In \cite{TW2022}, the authors employed also a Lyapunov-type function to prove the global existence of a maximal solution, but, differently from here, they construct the maximal solution as a strong limit of approximate solutions. The existence of these approximate solutions is based on a regularization method that avoids the compact embedding $E_1\subset \subset E_0$ but requires the drift and the noise coefficient to be (asymptotically) quasi-monotonic, a property not assumed in our paper. While they proposed a non-explosion criterion for a general framework (\cite[Theorems 2.1 and 3.1]{TW2022}), here we identify an explicit radial structure for the noise, with only one Brownian motion, and an explicit no blow-up condition for the $\sigma$ in the SDE (in \cite[Example 1.1]{RTW2020}, a different example of noise is proposed; in \cite{ALONSOORAN2022244} and \cite{RohTan2021} two examples with a one-dimensional Brownian motion are also used but for a one-dimensional transport equation and a Camassa-Holm-type equation). Moreover, our no blow-up condition Assumption \ref{Assump:sigma}-(c) (i.e.\ $|f(x)|^2\ge 2g(x)-K$) needs to hold only where the drift is large (that is where $g$ is large), so that the noise intensity does not need to be large.
\end{remark}

\begin{theorem}[Pathwise Uniqueness]\label{Unique}
Under Assumptions \ref{Assump:Spaces},~\ref{Assump:Projection},~\ref{Assump:drift}, and~\ref{Assump:sigma}, the pathwise uniqueness holds for the SDE \eqref{eq:SDE}. That is, if $X^{(1)}$ and $X^{(2)}$ are two solutions to \eqref{eq:SDE} on the same filtered probability space $(\Omega,\mathcal{A},(\mathcal{F}_t)_t,\Prob)$ with respect to the same Brownian motion $W$, with the same initial condition $X_0$, then
\begin{equation*}
\Prob\big\{X^{(1)}_t = X^{(2)}_t, \, \forall t\in [0,T]\big\}=1.
\end{equation*}
\end{theorem}

\begin{proof}[Proof of Theorem~\ref{Unique}]
Suppose that $X^{(1)}$, $X^{(2)}$ are two solutions to~\eqref{eq:SDE} with $X^{(1)}_0=X^{(2)}_0=X_0$. Then, the difference process $Y:=X^{(1)}-X^{(2)}$ satisfies the following equality in $E_{-1}$,
\begin{align*}
Y_t = \int_0^t (b(X^{(1)}_r)-b(X^{(2)}_r)) dr +\int_0^t (\sigma(X^{(1)}_r)-\sigma(X^{(2)}_r)) dW_r, \quad \forall t\in [0,T].
\end{align*}
Let $M>0$ and $\tau^i_M$ be the first exit time of $X^{(i)}$ from $B_{M;E_0}$, where $i=1,2$. Denote $\tau_M=\tau^1_M\wedge \tau^2_M$. Applying the It\^o formula to $\|Y_t\|^2_{E_{-1}}$ over the interval $[0,\tau_M]$, we obtain 
\begin{align*}
& \Vert Y_{t\wedge \tau_M}\Vert_{E_{-1}}^2 
= 
2\int_0^{t\wedge \tau_M} 
<Y_r, b(X^{(1)}_r)-b(X^{(2)}_r)>_{E_{-1}} \,dr
+\int_0^{t\wedge \tau_M} \Vert\sigma(X^{(1)}_r)-\sigma(X^{(2)}_r)\Vert_{E_{-1}}^2 \,dr
\\ & \qquad\qquad +
2\int_0^{t\wedge \tau_M} <Y_r,\sigma(X^{(1)}_r)-\sigma(X^{(2)}_r)>_{E_{-1}} \,dW_r, \quad \forall t\in [0,T].
\end{align*}
Within the interval $[0,\tau_M]$, we may use Assumption~\ref{Assump:drift}-(e) on the drift function and claim that
\begin{align*}
&
<Y_r, b(X^{(1)}_r)-b(X^{(2)}_r)>_{E_{-1}} \leq
k (\Vert X^{(1)}_r\Vert_{E_{0}}\vee\Vert X^{(2)}_r\Vert_{E_{0}} )\Vert Y_r\Vert_{E_{-1}}^2 \le C_M\Vert Y_r\Vert_{E_{-1}}^2
\end{align*}
for some positive constant $C_M$ depending on $M$. (Here and in what follows, we denote the constant coefficient in inequalities by $C$, with implicit dependence indicated if needed; we allow $C$ to change from line to line.) Similarly, by Assumption \ref{Assump:sigma}-(e) on the diffusion coefficient, we have
\begin{align*}
& \Vert\sigma(X^{(1)}_r)-\sigma(X^{(2)}_r)\Vert_{E_{-1}}^2 
\leq
h^2 (\Vert X^{(1)}_r\Vert_{E_{0}}\vee\Vert X^{(2)}_r\Vert_{E_{0}} )\Vert Y_r\Vert_{E_{-1}}^2 \le C_M\Vert Y_r\Vert_{E_{-1}}^2.
\end{align*}
Then, by the Burkholder-Davis-Gundy inequality, we conclude that for any $p\in [2,\infty)$,
\begin{align*}
&\E\left[
\sup_{s\in[0,t]}
\left\Vert 
Y_{s\wedge \tau_M}
\right\Vert^{2p}_{E_{-1}}
\right]\\
&\enspace \leq 
C_{M,p}\E\left[
\left(\int_0^{t\wedge \tau_M} 
\Vert Y_r\Vert_{E_{-1}}^2
\,dr\right)^p\right] +
 C_p\E\left[\left(
 \int_0^{t\wedge \tau_M} \left\Vert \sigma(X^{(1)}_r)-\sigma(X^{(2)}_r)\right\Vert_{E_{-1}}^2
 \Vert Y_r\Vert_{E_{-1}}^2
 \,dr
\right)^{p/2}
\right]\\
&\enspace\leq 
C_{M,p}\E\left[
\left(\int_0^{t\wedge \tau_M} 
\Vert Y_r\Vert_{E_{-1}}^2
\,dr\right)^p\right] +
 C_{M,p}\E\left[\left(
 \int_0^{t\wedge \tau_M}
 \Vert Y_r\Vert_{E_{-1}}^4
 \,dr
\right)^{p/2}
\right]\\
&\enspace\le C_{M,T,p}\E\left[
\int_0^{t\wedge \tau_M} 
\Vert Y_r\Vert_{E_{-1}}^{2p}
\,dr\right]
\le C_{M,T,p}
\int_0^t 
\E\left[\sup_{s\in [0,r]}\Vert Y_{s\wedge \tau_M}\Vert_{E_{-1}}^{2p}\right]
\,dr.
\end{align*}
Applying Gr\"{o}nwall's lemma, we arrive at
\begin{align*}
\E\left[
\sup_{s\in[0,t\wedge\tau_M]}
\left\Vert 
Y_s
\right\Vert_{E_{-1}}^{2p}
\right]
=
0.
\end{align*}
Since both $X^{(1)}$ and $X^{(2)}$ live in $C([0,T]; E_0)$ (see Definition~\ref{def:weak}), then, $\Prob$-a.s., $\tau^1_M\to T$ and $\tau^2_M\to T$ as $M\to\infty$, and then $Y_t=0$ at all $t\in [0,T]$ $\Prob$-a.s.. This ends the proof.
\end{proof}

\section{The Stochastic Euler Equations}

In this section, we address the stochastic Euler equation on a $d$-dimensional smooth (i.e.\ of class $C^\infty$), bounded, and simply connected domain $D$, with $d=2,3$:
\begin{subequations}\label{eq:stoch_Euler}
    \begin{gather}
  du +(u\cdot \nabla) u dt +\nabla p dt = \sigma(u)dW_t\text{ on }D,\label{eq:stoch_Euler_1}\\      \operatorname{div} u=0\text{ on }D,\label{divfree}\\
   u\cdot n=0\text{ on }\partial D.   
    \end{gather}
\end{subequations}
Here, $u:[0,T]\times D\times \Omega\to \R^d$ is the velocity field, $p:[0,T]\times D\times \Omega\to \R$ is the pressure field, and $W$ is a one-dimensional Brownian motion on a filtered probability space $(\Omega,\mathcal{A},(\mathcal{F}_t)_t,\Prob)$ satisfying the standard assumption. Besides, $n$ is the outer normal on $\partial D$, and the equation, $u\cdot n=0$, encodes the slip boundary condition. Both $u$ and $p$ are unknown functions, whereas the noise coefficient $\sigma$ will be specified later. 

Let $s\in\mathbb{N}$. We define the following (standard) function spaces 
\begin{equation*}
\mathcal{H}^s(D)=\left\{u\in H^s(D)\mid \operatorname{div} u=0\text{ on }D\text{ and }u\cdot n=0\text{ on }\partial D\right\},
\end{equation*}
where $H^s(D)$ refers to the set of functions whose weak derivatives up to order-$s$ belong to $L^2(D)$. We endow $\mathcal{H}^s(D)$ with the $H^s(D)$-scalar product. Note that $\mathcal{H}^s(D)$ is a separable Hilbert space for every positive integer $s$, as it is a closed subspace of the separable Hilbert space $H^s(D)$. With this notation, now we introduce $\sigma$.
\begin{assumption}\label{hp:noise_Euler}
The diffusion coefficient $\sigma:\mathcal{H}^s(D)\to \mathcal{H}^{s-1}(D)$ has the form of
\begin{align*}
\sigma(u)=c(1+\|u\|_{W^{1,\infty}(D)}^2)^{\beta/2} u
\end{align*}
for some $\beta\ge 1/2$ and $c>0$. If $\beta=1/2$, we assume additionally that $c>\sqrt{2\tilde C}$ for a specified constant $\tilde C$ (See~\eqref{eq:g(x)} below, note that $\tilde C$ depends only on $D$ and $s$, but not on the initial condition).
\end{assumption}

Under the Assumption~\ref{hp:noise_Euler} and the incompressibility condition~\eqref{divfree}, we can recover the pressure function $p$ out of~\eqref{eq:stoch_Euler}. Clearly, $p$ solves
\begin{align*}
  -\Delta p &= \operatorname{div} \big((u\cdot \nabla) u\big)\text{ on } D.
\end{align*}
Recall the Leray projector~ $\mathcal{P}$, which is the orthogonal projection of $L^2(D)$ onto its closed subspace $\mathcal{H}^0(D)$. Equivalently, for every $v\in L^2(D)$, we have $\mathcal{P}v=v-\mathcal{Q}v$, where $\mathcal{Q}v=-\nabla w$, and $w$ is uniquely defined up to an additive constant by the elliptic Neumann problem
\begin{align*}
  -\Delta w &= \operatorname{div} v\text{ on } D,\\
  \frac{\partial w}{\partial n} &=v\cdot n \text{ on } \partial D.
\end{align*}
Hence, applying $\mathcal{P}$ to~\eqref{eq:stoch_Euler_1}, we can eliminate the pressure gradient and interpret the stochastic Euler equations~\eqref{eq:stoch_Euler} as a stochastic evolution equation of type-\eqref{eq:SDE} on $\mathcal{H}^s(D)$ with $s\in \mathbb{N}$,
\begin{align*}
du = b(u)dt +\sigma(u)dW_t,
\end{align*}
where $\sigma(u)$ was defined in Assumption~\ref{hp:noise_Euler}, and 
\begin{equation*}
 b(u) = -\mathcal{P}[(u\cdot\nabla) u]. 
\end{equation*}

The main result of this section is the following:
\begin{theorem}\label{thm:ApplicationToEuler}
Let $s\in \mathbb{N}$ and $s>d/2+2$. Under the Assumption~\ref{hp:noise_Euler}, the stochastic Euler equations~\eqref{eq:stoch_Euler} have a unique, global-in-time, strong solution in $\mathcal{H}^{s}(D)$ for every initial condition $u_0$ in $\mathcal{H}^{s+1}(D)$.
\end{theorem}

\begin{remark}\label{rmk:Euler_optimality}
One may note that the initial condition is required to be in the smaller space $\mathcal{H}^{s+1}(D)$ rather than $\mathcal{H}^{s}(D)$, and that we assume $s>d/2+2$ rather than $s>d/2+1$. We expect that these assumptions could be weakened by a finer analysis. In fact, \cite{GlaVic2014} show local well-posedness with general diffusion coefficients, with finer assumptions on the initial condition and on $s$. However, their result does not apply here, since the noise $\sigma(u)=c(1+\|u\|_{W^{1,\infty}(D)}^2)^{\beta/2}u$ does not satisfy the assumptions of \cite[Theorem 4.3]{GlaVic2014}.
\end{remark}

\begin{remark}\label{rmk:TW2}
    \cite[Theorem 4.1]{TW2022} provides a non-explosion criterion for the stochastic Euler equations. Here, we identify a radial noise with a one-dimensional Bronwian motion and with the explicit diffusion coefficient~\eqref{noise:superlinear} avoiding explosion in Euler equations. Moreover, by requiring the no blow-up condition (Assumption \ref{Assump:sigma}-(c)) only where the drift is large, we can include $\beta>1/2$ in the construction~\eqref{noise:superlinear} without requiring the diffusion intensity $c$ in \eqref{noise:superlinear} to be large. This is a situation not discussed in \cite{TW2022} (the noise intensity is allowed not to be large in \cite{ALONSOORAN2022244} and \cite{RohTan2021} but for different SPDEs).
\end{remark}

To prepare for the proof, we first verify a few (classical and standard) conditions.

\begin{lemma}\label{Euler:embed}
 Let $s\in \mathbb{N}$ and $s>1$. Then, $\mathcal{H}^{s+1}(D)$ is compactly embedded in $\mathcal{H}^{s}(D)$, and  $\mathcal{H}^{s}(D)$ is continuously embedded in $\mathcal{H}^{s-1}(D)$. Both embeddings are dense. Moreover, there exist $\theta\in (0,1)$ and $C>0$ such that 
 \begin{equation*}
  \|v\|_{H^{s}(D)}\le C\|v\|_{H^{s+1}(D)}^{1-\theta}\|v\|_{H^{s-1}(D)}^\theta  , \quad \forall v\in \mathcal{H}^{s+1}(D).
 \end{equation*}
\end{lemma}

\begin{proof}[Proof of Lemma~\ref{Euler:embed}]
  First, the embeddings from $\mathcal{H}^{s+1}(D)$ to $\mathcal{H}^{s}(D)$ and $\mathcal{H}^{s}(D)$ to $\mathcal{H}^{s-1}(D)$ are continuous (by the definition of the space norms) and dense, because the Leray projection $\mathcal P: H^r(D)\to\mathcal{H}^r(D)$ is continuous for every positive integer $r$ (see~\cite[inequality~(2.9)]{GlaVic2014}) and $H^{s+1}(D)\hookrightarrow H^s(D)\hookrightarrow H^{s-1}(D)$ are dense embeddings (see~\cite[Theorem 3.17]{adams2003sobolev}). The compact embedding from $\mathcal{H}^{s+1}(D)$ to $\mathcal{H}^{s}(D)$ follows from the Rellich–Kondrachov theorem (see e.g.~\cite[Theorem 6.3]{adams2003sobolev})
  and the fact that divergence-free and slip boundary conditions are preserved in $H^s$-limits. Finally, the interpolation formula comes from \cite[Theorem 5.2]{adams2003sobolev}.
\end{proof}

\begin{lemma}\label{Euler:basis}
 Let $s\in \mathbb{N}$ and $s>1$. Then, $\mathcal{H}^{s-1}(D)$ has an orthonormal basis $\{e_k\}_{k\in \N_+}$ consisting of elements in $\mathcal{H}^{s+1}(D)$. Moreover, $\{e_k\}_{k\in \N_+}$ is also orthogonal in $\mathcal{H}^{s+1}(D)$.
\end{lemma}
\begin{proof}[Proof of Lemma~\ref{Euler:basis}]
We assert the following variational result: For any $w \in \mathcal{H}^{s-1}(D)$, there is a unique element $\Phi(w)$ in $\mathcal{H}^{s+1}(D)$ satisfying 
\begin{align}\label{lax-milgram}
    \langle \Phi(w),v\rangle_{H^{s+1}}=\langle w,v\rangle_{H^{s-1}} \text{ for all } v \in \mathcal{H}^{s+1}(D).
\end{align}

By the Cauchy-Schwartz inequality and the continuous embedding from $\mathcal{H}^{s+1}(D)$ to $\mathcal{H}^{s-1}(D)$, the linear functional $L_w(v):=\langle w,v\rangle_{H^{s-1}}$ is continuous in $\mathcal{H}^{s+1}(D)$. Meanwhile, the scalar product $a(u,v):=\langle u,v\rangle_{H^{s+1}}$ defines a bilinear form on $\mathcal{H}^{s+1}(D)\times \mathcal{H}^{s+1}(D)$ which is clearly coercive and continuous. Then, we may apply the Lax-Milgram theorem and conclude the existence of a unique element in $\mathcal{H}^{s+1}(D)$, say $\Phi(w)$, such that \eqref{lax-milgram} holds. Moreover, $\Phi: \mathcal{H}^{s-1}(D)\to\mathcal{H}^{s+1}(D)$ is injective, linear, and bounded. Given the compact embedding $\mathcal{H}^{s+1}(D)\subset\subset\mathcal{H}^{s-1}(D)$, the mapping $\Phi$ is actually compact on $\mathcal{H}^{s-1}(D)$. Since
\begin{align*}
    \langle \Phi(w),v\rangle_{H^{s-1}}
    =\langle \Phi(w),\Phi(v)\rangle_{H^{s+1}}
    =\langle w,\Phi(v)\rangle_{H^{s-1}}, \qquad \forall v,w \in \mathcal{H}^{s-1}(D),
\end{align*}
then $\Phi$ is also self-adjoint. By the spectral theorem, there exists an orthonormal basis $\{e_k\}_{k\in \N_+}$ of $\mathcal{H}^{s-1}(D)$ that are eigenfunctions of $\Phi$. The basis vectors belong to $\mathcal{H}^{s+1}(D)$ as well, and they are orthogonal with respect to both $H^{s-1}$ and $H^{s+1}$ scalar products (since they are eigenfunctions of $\Phi$). Then, Assumption \ref{Assump:Projection} holds with $C=1$ and $E^n=\Span\{e_1,\,...,\,e_n\}$. 
\end{proof}
\begin{remark}\label{basis:regular}
It follows from \cite[Theorem 4.1]{ghidaglia1984regularite} that $\{e_k\}_{k\in \N_+}$ actually live in $\mathcal{H}^{s+2}(D)$.  
\end{remark}

Next, we recall a few classical estimates (see also~\cite{GlaVic2014}). Let $m$ be an integer satisfying $m>d/2$. Then,
\begin{itemize}
    \item (The Moser Estimate)~$ \|u v\|_{H^{m}} \le C (\|u\|_{L^\infty}\|v\|_{H^{m}} +\|u\|_{H^{m}}\|v\|_{L^{\infty}})$,
    \item (The Sobolev Embedding)~$\|u\|_{L^\infty}\leq C\|u\|_{H^{m}}$ and $\|u\|_{W^{1,\infty}}\leq C\|u\|_{H^{m+1}}$.
\end{itemize}
Hence, for all integers $m$ such that $m>d/2$, we infer 
\begin{align}
    &\|(u\cdot\nabla)v\|_{H^{m}} \le C (\|u\|_{L^\infty}\|v\|_{H^{m+1}} +\|u\|_{H^{m}}\|v\|_{W^{1,\infty}})\le C \|u\|_{H^{m+1}}\|v\|_{H^{m+1}}\label{eq:bd_drift_1}
\end{align}
whenever $u,v\in \mathcal{H}^{m+1}(D)$. If $m>d/2+1$, then by~\cite[Lemma 2.1 (c)]{GlaVic2014},
\begin{align}
  \langle \mathcal{P}[(u\cdot \nabla)v], v\rangle_{H^m} \le C(\|u\|_{W^{1,\infty}}\|v\|_{H^{m}} +\|u\|_{H^{m}}\|v\|_{W^{1,\infty}})\|v\|_{H^m},\label{eq:bd_drift_2_2}
\end{align}
when $u\in \mathcal{H}^{m}(D)$ and $v\in \mathcal{H}^{m+1}(D)$. Note that \eqref{eq:bd_drift_2_2} is a result of the commutator estimates.

\begin{proof}[Proof of Theorem~\ref{thm:ApplicationToEuler}]
Given Lemmas~\ref{Euler:embed} and \ref{Euler:basis}, we are left to verify Assumptions~\ref{Assump:drift}--\ref{Assump:sigma} to apply the main result Theorem \ref{thm:exist_strong} to the stochastic Euler equations~\eqref{eq:stoch_Euler}. 

Assumption~\ref{Assump:drift}-(a) follows from the continuity of the Leray projector, the estimate~\eqref{eq:bd_drift_1} and a simple triangle inequality. 
Exploiting the orthogonality of $\{e_k\}_{k\in \N_+}$ in both $\langle \cdot,\cdot\rangle_{H^{s-1}}$ and $\langle \cdot,\cdot\rangle_{H^{s+1}}$-scalar products and applying~\eqref{eq:bd_drift_2_2} with $m=s+1$ and Remark~\ref{basis:regular}, we obtain for every $u\in E^n$ that
\begin{align*}
    \langle b_n(u), u \rangle_{H^{s+1}}
    = \langle \Pi_n \mathcal{P} (u\cdot\nabla u) , u \rangle_{H^{s+1}}
    = \langle \mathcal{P} (u\cdot\nabla u), u \rangle_{H^{s+1}}
    \le \tilde C \lVert u \rVert_{W^{1,\infty}} \lVert u \rVert_{H^{s+1}}^2.
\end{align*}
Hence, Assumption~\ref{Assump:drift}-(b) is met with
\begin{equation} \label{eq:g(x)}
    g(u)=\tilde C\lVert u \rVert_{W^{1,\infty}(D)}.
\end{equation}
Assumption \ref{Assump:drift}-(c) holds since $b_n$ is quadratic, and Assumption \ref{Assump:drift}-(d) is ensured by the hypothesis. Utilizing \eqref{eq:bd_drift_1}, \eqref{eq:bd_drift_2_2}, and the Sobolev embeddings, we have
\begin{align*}
   & \langle \mathcal{P}[v\cdot \nabla v-w\cdot \nabla w],v-w\rangle_{H^{s-1}}
    = \langle \mathcal{P}[v\cdot \nabla(v-w)],v-w\rangle_{H^{s-1}} 
     + \langle \mathcal{P}[(v-w)\cdot \nabla w],v-w\rangle_{H^{s-1}}
     \\ & \qquad\qquad\qquad  
     \leq C
     (\|v\|_{W^{1,\infty}}\|v-w\|_{H^{s-1}} +\|v\|_{H^{s-1}}\|v-w\|_{W^{1,\infty}})\|v-w\|_{H^{s-1}}
      \\ & \qquad\qquad\qquad\quad  
      +C
     (\|v-w\|_{L^{\infty}}\|w\|_{H^{s}} +\|v-w\|_{H^{s-1}}\|w\|_{W^{1,\infty}})\|v-w\|_{H^{s-1}}
     \\ & \qquad\qquad \qquad  
     \leq
      C (\lVert v\rVert_{H^{s-1}} +\lVert w\rVert_{H^{s}} )\lVert v-w\rVert_{H^{s-1}}^2,
\end{align*}
which fulfils Assumption \ref{Assump:drift}-(e) if $s>d/2+2$.

Besides, by our choice of $\sigma(u)$ in Assumption~\ref{hp:noise_Euler}, $\sigma(u)=f(u)u$ for all $u\in \mathcal{H}^{s+1}(D)$ with
\begin{align*}
f(u) = c(1+\|u\|_{W^{1,\infty}(D)}^2)^{\beta/2}.
\end{align*}
The function $f$ is clearly bounded on $\mathcal{H}^{s+1}(D)$-balls by the Sobolev embedding. Moreover, for the $g(u)$ in~\eqref{eq:g(x)},
\begin{equation*}
    |f(u)|^2= c^2\left(1+\frac{|g(u)|^2}{\tilde C^2}\right)^{\beta}>
    \begin{cases}
     \frac{c^2}{\tilde C}|g(u)| &\qquad \text{if }\beta=\frac{1}{2}, \\
      c^2|g(u)|^{2\beta}/\tilde C^{2\beta} &\qquad \text{if }\beta>\frac{1}{2}.
    \end{cases}
\end{equation*}
Then, Assumption~\ref{Assump:sigma}-(c) holds with $K=0$ and a sufficiently large number $G$ if $\beta>1/2$, or with $K=0$ and an arbitrary positive number $G$ if $\beta=1/2$ and $c^2>2\tilde C$. Given such a prescribed $G$, Assumption~\ref{Assump:sigma}-(b) is satisfied with $L=c(1+G^2/\tilde C^2)^{\beta/2}$. Assumption~\ref{Assump:sigma}-(e) also holds. That is because the function $s\mapsto (1+s^2)^{\beta/2}$ is infinitely differentiable, the space $H^{s-1}(D)$ is continuously embedded into $W^{1,\infty}(D)$, and then
\begin{align}
\begin{split}\label{sigma:cont}
    &\lVert \sigma (v) - \sigma(w) \rVert_{H^{s-1}} \\
    &\quad\le c \left( 1+\lVert v \rVert^2_{W^{1,\infty}} \right)^\frac{\beta}{2} \lVert v-w \rVert_{H^{s-1}} 
    + c \lVert w\rVert_{H^{s-1}} \left| \left(1+\lVert v \rVert^2_{W^{1,\infty}}\right)^\frac{\beta}{2} - \left(1+\lVert w \rVert^2_{W^{1,\infty}}\right)^\frac{\beta}{2}\right|\\
    &\quad \le h \left(\Vert v\Vert_{H^{s}}\vee\Vert w\Vert_{H^{s}} \right)  \Vert v-w\Vert_{H^{s-1}},\quad \forall v,w\in \mathcal{H}^{s}(D) ,
    \end{split}
\end{align}
for an increasing function $h:\R\to \R_+$. Assumption~\ref{Assump:sigma}-(a) follows immediately from \eqref{sigma:cont}, and Assumption~\ref{Assump:sigma}-(d) is a consequence of Assumption~\ref{Assump:sigma}-(e) and the equivalence of $H^{s-1}$, $H^{s}$, and $H^{s+1}$-norms on each finite-dimensional space $E^n=\Span\{e_1,\,...,\,e_n\}$. 

This completes the proof. By Theorems~\ref{thm:exist_global} and \ref{Unique} in Section~\ref{sec4}, the stochastic Euler equations~\eqref{eq:stoch_Euler} have a unique, global-in-time, strong solution in $\mathcal{H}^{s}(D)$ if the initial condition $u_0\in \mathcal{H}^{s+1}(D)$ for $s>d/2+2$.
\end{proof}

\section{The case of Stratonovich noise}

In this Section we consider briefly possible no blow-up effects by Stratonovich noise. We show that a one-dimensional Stratonovich noise with radial diffusion coefficient does not avoid blow-up (actually, it can induce blow-up). In Remarks \ref{rmk:Strat_positive_Bagnara} and \ref{rmk:Strat_positive_Ito}, we recall positive examples from the work \cite{Bagnara2023} in the Stratonovich case with a more complex noise. 

It is natural to ask whether Theorem \ref{thm:exist_strong} still holds when the It\^o integration is replaced by the Stratonovich one. Unfortunately, as the following example suggests, counterexamples arise even in one dimension, where the nonlinear diffusion itself can cause blow-up.

Let $X_t\in \R$ be a solution of
\begin{align}\label{eq:example_1D_Strat}
    \dd X_t= X^2_t \circ \dd W_t.
\end{align}
By converting the SDE in It\^o form, $X_t$ solves $\dd X_t=X_t^3\dd t + X^2_t \dd W_t$, which admits an explicit solution that exhibits blow-up \cite[Exercise 9.23]{Baldi2017}.

The next result shows that the above counterexample is not special: in dimension $1$, a one-dimensional Stratonovich noise cannot prevent blow-up, actually it induces blow-up. This result is a slight modification of \cite[Proposition 5.1]{maurelli2020non} (in fact inspired by \cite{Sch1995}), which can be generalized to higher, possibly infinite dimension, see the discussion after the proof below.

Consider the SDE on $\R$: 
\begin{align}\label{eq:Strat_SDE}
    \dd X = b(X) \dd t + \sigma(X)\circ \dd W,
\end{align}
where $b:\R\to \R$ and $\sigma:\R\to \R$ are given functions and $W$ is a one-dimensional Brownian motion.

\begin{proposition}\label{prop:Strat_blowup}
    Assume that $b:\R\to \R$ is locally Lipschitz, $\sigma:\R\to \R$ is $C^1$ with locally Lipschitz derivative, and $W$ is a one-dimensional Brownian motion (on a filtered probability space with standard assumption). Assume that, for some $c>0$, $\sigma(x)\ge c$ for every $x>0$. Assume also at least one of the following:
    \begin{itemize}
        \item Explosive drift: $b(x)>0$ for $x>0$ and
        \begin{align*}
            \int_0^\infty  \frac{1}{b(x)} \dd x <\infty.
        \end{align*}
        \item Superlinear diffusion:
        \begin{align*}
            \int_0^\infty \frac{1}{\sigma(x)} \dd x <\infty,\qquad \int_0^\infty \frac{b^-(x)}{\sigma(x)^2} \dd x <\infty,
        \end{align*}
        where $b^-(x)=\max\{-b(x),0\}$.
    \end{itemize}
    Then for every $x_0>0$, the solution $X$ to \eqref{eq:Strat_SDE} with initial condition $x_0$ blows up in finite time with positive probability.
\end{proposition}

Note that, in the explosive drift case, the solution to the ODE $\dot{Z}=b(Z)$ blows up in finite time.

\begin{proof}
    We define the transformation
    \begin{align*}
        \Phi(x) = \int_0^x \frac{1}{\sigma(x')} \dd x',\quad 0<x<\infty.
    \end{align*}
    By It\^o formula, $Y=\Phi(X)$ solves the transformed SDE
    \begin{align*}
        \dd Y = A(Y) \dd t +\dd W,\quad \text{with } A(y)=\frac{b(\Phi^{-1}(y))}{\sigma(\Phi^{-1}(y))},
    \end{align*}
    up to the exit time $S$ of $Y$ from $(0,\Phi(\infty))$. Blow-up for $X$ coincides with the event $\{S<\infty,\, Y_S=\Phi(\infty)\}$. By Feller's test, in the version for example of \cite[ Theorem 5.29, Proposition 5.22 and Problem 5.27]{KasShr1991}, the latter event has positive probability provided that
    \begin{align}\label{eq:Feller}
        \int_0^{\Phi(\infty)} \int_0^x \exp\left(-2 \int_y^x A(z) \dd z\right) \dd y \dd x <\infty.
    \end{align}
    
    We consider first the case of explosive drift. In this case, if $\Phi(\infty)<\infty$, then \eqref{eq:Feller} is clearly satisfied since $A$ is positive; if $\Phi(\infty)=\infty$, then by \cite[Corollary 2]{Sch1995}
    \begin{align*}
        2\int_0^{\Phi(\infty)} \int_0^x \exp\left(-2 \int_y^x A(z) \dd z\right) \dd y \dd x \le \int_0^\infty \frac{1}{A(y)} \dd y = \int_0^\infty \frac{1}{b(y)} \dd y < \infty,
    \end{align*}
    so \eqref{eq:Feller} is satisfied. Hence blow-up holds in the case of explosive drift.

    Now we consider the case of superlinear diffusion. In this case $\Phi(\infty)<\infty$ by assumption and, for $0<y<x<\Phi(\infty)$,
    \begin{align*}
        -2 \int_y^x A(z) \dd z = -2 \int_{\Phi^{-1}(y)}^{\Phi^{-1}(x)} \frac{b(z)}{\sigma(z)^2} \dd z \le 2 \int_0^\infty \frac{b^-(z)}{\sigma(z)^2} \dd z<\infty.
    \end{align*}
    Hence \eqref{eq:Feller} is satisfied and blow-up holds also in the case of superlinear diffusion. The proof is complete.
\end{proof}

The result above can be generalized to SDEs in higher, possibly infinite dimension, with a one-dimensional Stratonovich noise and with a diffusion coefficient whose radial component depends only on the norm of the solution; we just give a sketch of the argument. Namely, one can consider an SDE on $\R^d$ or on a Hilbert space $H$ of the form
\begin{align*}
    \dd X = b(X) \dd t +\sigma(X) \circ \dd W
\end{align*}
and study the evolution of the norm $\|X\|$. If, for example, outside a ball we have $b(x)\cdot x/\|x\|\ge b_{rad}(\|x\|)$ and $\sigma(x)\cdot x/\|x\| = \sigma_{rad}(\|x\|)$ for suitable functions $b_{rad}$, $\sigma_{rad}$, we can compare $\|X\|$ with the solution $R$ to
\begin{align*}
    \dd R = b_{rad}(R) \dd t +\sigma_{rad}(R) \circ \dd W
\end{align*}
and apply Proposition \ref{prop:Strat_blowup} to show that, for suitable (but quite general) $b_{rad}$ and $\sigma_{rad}$, $R$ and hence $X$ blow up in finite time with positive probability.

\begin{remark}\label{rmk:Strat_positive_Bagnara}
    We may ask whether other forms of Stratonovich noises can prevent solutions to SPDEs from blowing up. The result is positive in some cases: as demonstrated by the first author in \cite{Bagnara2023} some months after the completion of this manuscript, a no blow-up result can also be obtained in the Stratonovich framework by employing a different diffusion coefficient and requiring the noise to be at least two dimensional.
\end{remark}

\begin{remark}\label{rmk:Strat_positive_Ito}
    A further question is whether there are noises that avoid blow-up both with It\^o and Stratonovich integration. So far, we are not aware of positive examples. Indeed, the strategy used in \cite{Bagnara2023} for Stratonovich case does not easily go through in the It\^o case, so it is unclear whether the same no blow-up result in \cite{Bagnara2023} works also with It\^o noise. We leave this question for future investigation.
\end{remark}

\section{Appendix}

\begin{proof}[Proof of Lemma~\ref{BM}]
Since $\tilde W^n$ has the same probability law as $W$, it is a Brownian motion with respect to its natural filtration. Then, $\tilde W^n_t$ is $\tilde{\mathcal{G}}^n_t$-measurable for all $t\in[0,T]$. Moreover, since $(\tilde X^n, \tilde W^n)$ has the same law as $( X^n, W)$, the independence of $W_t-W_s$ from $\mathcal{\sigma}\{X^n_r,W_r\mid r\le s\}$ implies the independence of $\tilde{W}^n_t-\tilde{W}^n_s$ from $\tilde{\mathcal{G}}^n_s$ whenever $s\le t$. Hence, $\tilde W^n$ is a $(\tilde{\mathcal{G}}^n_t)_{t\in[0,T]}$-adapted Brownian motion.

As the almost sure limit of a sequence of Brownian motions, $\tilde W$ is also a Brownian motion. Certainly, $\tilde W_t$ is $\tilde{\mathcal{G}}_t$-measurable for every $t\in[0,T]$ because $ \mathcal{\sigma}\{\tilde W_s\mid s\le t\}\subset \tilde{\mathcal{G}}_t$. Note that $\tilde W^n_t - \tilde W^n_s$ is independent of all tuples $(\tilde W^n_{s_1},\tilde X^n_{s_1},...,\tilde W^n_{s_k},\tilde X^n_{s_k})$ with $0\le s_1 \le ... \le s_k\le s <t$. Passing to the a.s.\ limit, we conclude that $\tilde{W}_t-\tilde{W}_s$ is independent of $\mathcal{\sigma}\{\tilde X_r,\tilde W_r\mid r\le s\}$ and $(\tilde{\mathcal{G}}_r)_{r\le s}$. Hence, $\tilde W$ is a $(\tilde{\mathcal{G}}_t)_{t\in[0,T]}$-adapted Brownian motion.

For the extension of the results from $\tilde{\mathcal{G}}^n_s, \tilde{\mathcal{G}}_s$ to $\tilde{\mathcal{F}}^n_s, \tilde{\mathcal{F}}_s$, we resort to the argument for proving~\cite[Proposition 2.5]{bass2011stochastic}.
\end{proof}

\begin{proof}[Proof of Lemma~\ref{WS}]
It is equivalent to verify that, for every $t\in [0,T]$, the random variable 
\begin{equation*}
    \tilde Z^n_t \coloneqq \tilde X^n_t-\tilde X^n_0 -\int_0^t b_n(\tilde X^n_s)\,ds-\int_0^t\sigma_n(\tilde X^n_s)\,d\tilde W^n_s
\end{equation*}
is equal to zero $\widetilde\Prob$-a.s.. We denote by $Z^n_t$ the corresponding random variable when $X^n, W$ are present in place of $\tilde X^n, \tilde W^n$. Since $Z^n_t=0$ almost everywhere in $[0,T]\times\Omega$, we turn to prove that $\tilde Z^n_t$ and $Z^n_t$ have the same probability law at all  $t\in [0,T]$. 

As a first step, we show that the discretized processes of $\tilde Z^n_t$ and $Z^n_t$ are distributed the same. We introduce
\begin{align*}
   & \tilde Q^{m,N}_t \coloneqq \tilde X^n_t-\tilde X^n_0 -\sum_{i=0}^{N-1}
   \int_{t^N_i}^{t^N_{i+1}}
   b_n(\tilde X^n_{t^N_i})
   \mathds{1}_{[0, \tilde \tau^n_m)}(s)\,ds
   -
   \sum_{i=0}^{N-1}
   \int_{t^N_i}^{t^N_{i+1}} 
   \sigma_n(\tilde X^n_{t^N_i}) \mathds{1}_{[0, \tilde \tau^n_m)}(s)\,d\tilde W^n_s,\\
    & Q^{m,N}_t \coloneqq  X^n_t- X^n_0 -
    \sum_{i=0}^{N-1}
    \int_{t^N_i}^{t^N_{i+1}} b_n( X^n_{t^N_i})
    \mathds{1}_{[0,  \tau^n_m)}(s)\,ds-
    \sum_{i=0}^{N-1}
    \int_{t^N_i}^{t^N_{i+1}} \sigma_n( X^n_{t^N_i}) 
 \mathds{1}_{[0,\tau^n_m)}(s)\,d W_s,
\end{align*}
where $t^N_i=it/N$ with $i\in\{0, \ldots, N\}$ and the stopping times $\tilde \tau^n_m, \tau^n_m$ are defined as follows,
\begin{equation*}
    \tilde \tau^n_m \coloneqq \inf\left\{ t\in[0,T]\,:\, \lVert \tilde X^n_s\rVert_{E_{0}}\ge m  \right\},
    \qquad
     \tau^n_m \coloneqq \inf\left\{ t\in[0,T]\,:\, \lVert X^n_s\rVert_{E_{0}}\ge m  \right\}
\end{equation*}
for every $m\in\mathbb{N}$. Clearly, $\tilde Q^{m,N}_t $ and $Q^{m,N}_t $ are Borel function respectively of $(\tilde X^n, \tilde W^n)$ and of $( X^n, W)$ (as random variables in $C\left([0,T];E_0\right)\times C\left([0,T];\R\right)$). Since $(\tilde X^n, \tilde W^n)$ and $( X^n, W)$ have the same law as random variables in $C\left([0,T];E_0\right)\times C\left([0,T];\R\right)$, then $\tilde Q^{m,N}_t$ is distributed the same as $ Q^{m,N}_t$ for every positive integer $m$ and $N$.

In the remaining part of the proof, we show that both $ Q^{m,N}_t$ and $\tilde Q^{m,N}_t$ have an almost-sure limit as we pass $N\to \infty$ along some subsequence and then pass $m\to \infty$; The limits are $Z^n_t$ and $\tilde Z^n_t$ respectively, and consequently, $Z^n_t$ and $\tilde Z^n_t$ have the same distribution.

Indeed, we may employ the It\^o isometry and conclude that
\begin{align*}
   & 
   \E\left[\left\Vert
   \sum_{i=0}^{N-1}
   \int_{t^N_i}^{t^N_{i+1}} 
   \sigma_n( X^n_{t^N_i})\mathds{1}_{[0,  \tau^n_m)}(s)\,d W_s
   -
    \int_{0}^{t} \sigma_n(  X^n(s)) \mathds{1}_{[0,  \tau^n_m)}(s)\,d  W_s
 \right\Vert_{E_{-1}}^2\right]
   \\& \quad 
    =
   \E\left[
   \int_{0}^{t} 
    \big\Vert\sum_{i=0}^{N-1}
    \mathds{1}_{[t^N_i,  t^N_{i+1})}(s)\sigma_n( X^n_{t^N_i})
   -
   \sigma_n(  X^n(s))
   \big\Vert_{E_{-1}}^2
   \mathds{1}_{[0,  \tau^n_m)}(s)
   \,ds
   \right]
   .
\end{align*}
Recall Assumption~\ref{Assump:Projection} and Assumption~\ref{Assump:sigma}~(a), which infers the continuity and boundedness of $\sigma_n$ on the $E_0$-ball defined by $\tau^n_m$. Hence, for every $s\in[0,t\wedge \tau^n_m]$ and every $\omega\in\Omega$,
\begin{align*}
    \big\Vert\sum_{i=0}^{N-1}
    \mathds{1}_{[t^N_i,  t^N_{i+1})}(s)\sigma_n( X^n_{t^N_i})
   -
   \sigma_n(  X^n(s))
   \big\Vert_{E_{-1}}
   \longrightarrow 0,
   \end{align*}
as $N\to\infty$, and this norm of difference is uniformly bounded in $\Omega\times [0, T]$ and for all $N$. Therefore, by the dominated convergence theorem,
\begin{equation*}
   \mathcal{I}^{m,N}:=  \sum_{i=0}^{N-1}
   \int_{t^N_i}^{t^N_{i+1}} 
   \sigma_n( X^n_{t^N_i})\mathds{1}_{[0,  \tau^n_m)}(s)\,d W_s
   \longrightarrow
    \mathcal{I}^{m}:= \int_{0}^{t} \sigma_n(  X^n(s)) \mathds{1}_{[0,  \tau^n_m)}(s)\,d  W_s
\end{equation*}
in $L^2\left(\Omega; E_{-1}\right)$ as $N\to \infty$. For the drift term, we resort to Assumption~\ref{Assump:drift}~(a) and the dominated convergence theorem, concluding that
\begin{align*}
   & \left\Vert
    \sum_{i=0}^{N-1}
   \int_{t^N_i}^{t^N_{i+1}} 
   b_n( X^n_{t^N_i})\mathds{1}_{[0,  \tau^n_m)}(s)\,ds
   -
   \int_{0}^{t} b_n(  X^n(s)) \mathds{1}_{[0,  \tau^n_m)}(s)\,ds
   \right\Vert_{E_{-1}}
   \\ & \quad 
   \leq 
   \int_{0}^{t} 
    \big\Vert\sum_{i=0}^{N-1}
    \mathds{1}_{[t^N_i,  t^N_{i+1})}(s) b_n( X^n_{t^N_i})
   -
   b_n(  X^n(s))
   \big\Vert_{E_{-1}}
   \mathds{1}_{[0,  \tau^n_m)}(s)
   \,ds
    \longrightarrow 0, \quad \mathbb{P}\text{-a.s.}
    \end{align*}
 as $N\to \infty$. By extracting a subsequence of $Q^{m,N}_t$, say $Q^{m,N_k}_t$, we have $\mathcal{I}^{m, N_k}\to \mathcal{I}^{m}$ $\mathbb{P}$-a.s., and thus,
\begin{equation*}
\lim_{k\to\infty} Q^{m,N_k}_t=
X^n_t- X^n_0 -\int_{0}^{t} b_n(  X^n(s))\mathds{1}_{[0,  \tau^n_m)}(s) \,ds   -\int_{0}^{t} \sigma_n(  X^n(s)) \mathds{1}_{[0,  \tau^n_m)}(s)\,d  W_s,\quad \mathbb{P}\text{-a.s.}
\end{equation*}
Note that $\lim_{m\to\infty}\tau^n_m=T$ due to Lemma~\ref{Lemma:BoundednessInProb} and the continuous embedding from $E_1$ to $E_0$. Then,
\begin{equation*}
\lim_{m\to\infty}\lim_{k\to\infty} Q^{m,N_k}_t=
Z^n_t,\quad \mathbb{P}\text{-a.s.}
\end{equation*}

Applying a similar argument to $\tilde Q^{m,N}_t$, we also obtain $ \tilde Z^n_t$ as a
$\widetilde{\mathbb{P}}$-almost sure limit of $\tilde Q^{m,N}_t$ for every $t\in[0,T]$. Since the almost sure convergence implies convergence in distribution, we conclude that $\tilde Z^n_t$ and $Z^n_t$ have the same law, which completes the proof.
\end{proof}

\begin{proof}[Proof of Lemma~\ref{lemma:StochIntConvergence}]
For convenience, we denote
\begin{align*}
\mathcal{I}^n(t) :=\int_0^t  G^n\, d W^n,&\qquad 
\mathcal{I}(t):=\int_0^t  G\, d W;\\
\mathcal{I}^n_N(t) :=\sum_{k=1}^{N}\int_0^t  G_k^n\, d W_k^n,&\qquad 
\mathcal{I}_N(t):=\sum_{k=1}^{N}\int_0^t  G_k\, d W_k,
\end{align*}
and split the difference $\mathcal{I}^n- \mathcal{I}$ as follows,
\begin{align}\label{split}
&\mathcal{I}^n- \mathcal{I} =
  (\mathcal{I}^n- \mathcal{I}^n_N  )
    +(\mathcal{I}^n_N- \mathcal{I}_N)
      +(\mathcal{I}_N- \mathcal{I}  ).
\end{align}
Recall that for any positive numbers $a, b$ and progressively measurable process $\Phi$ that lives in $L^{2}([0,T]; L_{2}(\mathfrak{U}, \mathcal H))$ $\Prob$-a.s., we have
\begin{equation}\label{ineq:StochIntProb}
    \Prob \left( \sup_{t\in[0,T]}\left\Vert \int_0^t \Phi_s \,dW_s  \right\Vert_{\mathcal H}>a \right) \le \frac b{a^2} + \Prob \left( \int_0^T  \lVert \Phi_s \rVert^{2}_{L_{2}(\mathfrak{U},\mathcal H)}\,ds>b\right),
\end{equation}
(see e.g.~\cite[Proposition 4.31]{da2014stochastic}\footnote{In \cite{da2014stochastic}, integrands are assumed to be predictable, but this is not a restriction since every progressively measurable process which is $\Prob$-a.s.\ in $L^1([0,T];L_2(\mathfrak{U}, \mathcal H))$ has a predictable version, see e.g.\ the argument in [Kunita, Stochastic flows and stochastic differential equations, page 60].}). Take $\epsilon>0$ and $\delta>0$. Exploiting the inequality~\eqref{ineq:StochIntProb} with $a=\epsilon$ and $b=\delta \epsilon^2$ for $\mathcal{I}^{n}-\mathcal{I}_N^n$, we obtain
\begin{align*}
  &\Prob \left(\sup_{t\in[0,T]}\left\Vert \mathcal{I}^{n}(t) -    \mathcal{I}_N^n(t)  \right\Vert_{\mathcal H} >\epsilon \right)
        \le \delta +    
          \Prob \left(  \sum_{k= N+1}^{\infty}\int_{0}^{T} \left\Vert G^{n}_{k} \right\Vert_{\mathcal H}^{2} dt >  \delta \epsilon^{2} \right)\\
          &\quad\leq \delta +    
          \Prob \left(  \int_{0}^{T} \left\Vert G^{n} - G \right\Vert^{2}_{L_{2}(\mathfrak{U}, \mathcal H)}  dt >  \frac{\delta \epsilon^{2}}{4} \right)+ \Prob \left(   \sum_{k= N+1}^{\infty} \int_{0}^{T}  \left\Vert G_k\right\Vert_{\mathcal H}^{2} dt >  \frac{\delta \epsilon^{2}}{4} \right).
\end{align*}
Since $G\in L^2([0,T]; L_2(\mathfrak{U}, \mathcal H))$ $ \Prob$-a.s., there is a sufficiently large $N$ independent of $n$, such that the last term in the above inequality is smaller than $\delta/2$. Hence,
\begin{align*}
    \Prob \left(\sup_{t\in[0,T]}\left\Vert \mathcal{I}^{n}(t) -    \mathcal{I}_N^n(t)  \right\Vert_{\mathcal H} >\epsilon \right) \le \frac{3\delta}{2} +\Prob \left(  \int_{0}^{T} \left\Vert G^{n} - G \right\Vert^{2}_{L_{2}(\mathfrak{U}, \mathcal H)}  dt >  \frac{\delta \epsilon^{2}}{4} \right).
\end{align*}
By the assumption $G^n\to G$ in $L^2([0,T]; L_2(\mathfrak{U}, \mathcal H))$ as $n\to\infty$ in probability, we can choose $n_1$, depending on $\delta$ and $\epsilon$, such that, for every $n\ge n_1$, the last term in the above inequality is smaller than $\delta/2$. Hence, for every $n\ge n_1$, we arrive at
\begin{align}\label{eq:est1}
    \Prob \left(\sup_{t\in[0,T]}\left\Vert \mathcal{I}^{n}(t) -    \mathcal{I}_N^n(t)  \right\Vert_{\mathcal H} >\epsilon \right) \le 2\delta.
\end{align}
Proceeding similarly for $\mathcal{I}-\mathcal{I}_N$ and choosing a large enough $N$, we obtain
\begin{align}\label{est-2}
    \Prob \left(\sup_{t\in[0,T]}\left\Vert \mathcal{I}(t)-\mathcal{I}_N(t)  \right\Vert_{\mathcal H} >\epsilon \right)
          &\le \delta +    
          \Prob \left(  \sum_{k= N+1}^{\infty} \int_{0}^{T} \left\Vert G_{k} \right\Vert_{\mathcal H}^{2} dt >  \delta \epsilon^{2} \right) \le 2\delta.
\end{align}

With such large $N$, now we estimate $\mathcal{I}_N^n- \mathcal{I}_N$ in the splitting~\eqref{split}. Note that $\mathcal{I}_N^n- \mathcal{I}_N$ is a finite sum of stochastic integrals with respect to different Brownian motions $W^n, W$, which we expect to handle
by exploiting an integration by parts formula and the convergence of $W^n$ to $W$ in $C([0,T])$. For this reason, we introduce a smoothing operator $\mathcal{R}_{\rho}$ for every $\rho>0$, 
\begin{equation}\label{eq:smoothingOperator}
  \mathcal{R}_{\rho}\Phi\,(t) 
  := \frac{1}{\rho} \int_0^t 
  \exp\left( - \frac{t -s}{\rho} \right) \Phi(s)\, ds,\quad \Phi\in L^{1}([0,T], \mathcal H).
\end{equation}
The operators $\mathcal{R}_{\rho}$ are approximations of identity (though with a nonsmooth function), hence the following properties hold (see e.g. \cite[estimates (4.19)]{bensoussan1995}, and \cite[Theorem 2.29]{adams2003sobolev} for a similar proof),
\begin{align}  \begin{split}\label{eq:smoothing_properties}
        &\lVert \mathcal{R}_{\rho}\Phi \rVert_{L^2([0,T];\mathcal H)} \le \lVert  \Phi\rVert_{L^2([0,T];\mathcal H)},\\
        &\lVert \mathcal{R}_{\rho}\Phi - \Phi \rVert_{L^2([0,T];\mathcal H)} \to 0  \text{ as $\rho \to 0$}.
    \end{split}
\end{align}
We split $\mathcal{I}_N^n- \mathcal{I}_N$ as follows:
\begin{align*}
    &\mathcal{I}_N^n(t)- \mathcal{I}_N(t)
    = \sum_{k=1}^N \int_0^t ( G^n_k - \mathcal{R}_\rho G^n_{k} )dW^n_k
    +\sum_{k=1}^N \left(\int_0^t \mathcal{R}_\rho G^n_{k} dW^n_k
      -\int_0^t \mathcal{R}_\rho G_{k} dW_k\right) \\
      &\quad\qquad\qquad
      + \sum_{k=1}^N\int_0^t ( \mathcal{R}_\rho G_{k} - G_k) dW_k=: A_1(t)+A_2(t)+A_3(t).
\end{align*}
For the term $A_1$, we employ the inequality~\eqref{ineq:StochIntProb} with $a=\epsilon$, $b=\delta\epsilon^2$ and the operator properties~\eqref{eq:smoothing_properties}, obtaining
\begin{align}
 &\Prob\biggl( \sup_{t\in[0,T]}\Vert A_1(t) \Vert_{\mathcal H} >\epsilon \biggr)
 \leq \delta + \Prob \left(  \int_{0}^{T} \sum_{k=1}^N\left\Vert G^{n}_{k} - \mathcal{R}_{\rho} G^{n}_{k} \right\Vert^{2}_{\mathcal H} ds > \delta \epsilon^{2} \right)
   \nonumber \\
    &\quad \leq \delta  +
    \Prob\left(  \int_0^T \sum_{k=1}^N \left\Vert G_k^n - G_k\right\Vert^{2}_{\mathcal H} dt > \frac{\delta \epsilon^{2}}{9} \right) + 
    \Prob\left(  \int_0^T \sum_{k=1}^N \left\Vert G_k - \mathcal{R}_\rho G_{k} \right\Vert^{2}_{\mathcal H} dt > \frac{\delta \epsilon^{2}}{9} \right) 
   \nonumber \\
    &\qquad + \Prob\left(  \int_0^T  \sum_{k=1}^N \left\Vert \mathcal{R}_\rho G_{k} - \mathcal{R}_\rho G^n_{k} \right\Vert^{2}_{\mathcal H} dt > \frac{\delta \epsilon^{2}}{9} \right)
    \nonumber\\
    &\quad\leq \delta +
    2\Prob\left(  \int_0^T \sum_{k=1}^N\left\Vert G_k^n - G_k\right\Vert^{2}_{\mathcal H} dt > \frac{\delta \epsilon^{2}}{9} \right)+
    \Prob\left(  \int_0^T \sum_{k=1}^N\left\Vert G_k - \mathcal{R}_\rho G_{k} \right\Vert^{2}_{\mathcal H} dt > \frac{\delta \epsilon^{2}}{9} \right).\nonumber
\end{align}
Since $\Vert\mathcal{R}_{\rho} G_{k}- G_k\Vert_{L^2([0,T];\mathcal H)}\to 0$ (as $\rho\to 0$) $\Prob$-a.s.\ for every $k\in\mathbb{N}$, the convergence is also in probability. Hence, there exists a small $\rho>0$ which depends on $N$ (but not on $n$), such that, for every $n$, the last term in the above inequality is smaller than $\delta$. We arrive at
\begin{align*}
    \Prob\biggl( \sup_{t\in[0,T]}\Vert A_1(t) \Vert_{\mathcal H} >\epsilon  \biggr)
    \leq 2\delta +
    2\Prob\left(  \int_0^T \sum_{k=1}^N\left\Vert G_k^n - G_k\right\Vert^{2}_{\mathcal H} dt > \frac{\delta \epsilon^{2}}{9} \right).
\end{align*}
By the assumption~\eqref{eq:ConvAsmpStochIntegrand}, we can choose $n_2$, depending on $\delta$ and $\epsilon$, such that the last term in the above inequality is smaller than $\delta$ whenever $n\ge n_2$. Hence, for every $n\ge n_2$, we have
\begin{align*}
    \Prob\biggl( \sup_{t\in[0,T]}\Vert A_1(t) \Vert_{\mathcal H} >\epsilon  \biggr)
    \leq 4\delta.
\end{align*}
Proceeding similarly for $A_3$, we conclude that for a sufficiently small $\rho>0$ which depends on $N$,
\begin{align*}
    \Prob\biggl( \sup_{t\in[0,T]}\Vert A_3(t) \Vert_{\mathcal H} >\epsilon  \biggr)
    &\leq \delta + \Prob \left(  \int_{0}^{T} \sum_{k=1}^N\left\Vert G_{k} - \mathcal{R}_{\rho} G_{k} \right\Vert^{2}_{\mathcal H} ds > \delta \epsilon^{2} \right)
    \le 2\delta.
\end{align*}
We are left to deal with the term $A_2$. For every $k$, since $\mathcal{R}_\rho G_k$ and $\mathcal{R}_\rho G^n_k$ have $W^{1,2}([0,T];\mathcal{H})$-trajectories $\Prob$-a.s., the integration by parts formula below holds without a stochastic correction: applying the It\^o formula (e.g.~\cite[Theorem 4.32]{da2014stochastic}) to $\mathcal{R}_\rho G_k\,d W_k$, we obtain
\begin{align*}
        &\int_0^t \mathcal{R}_\rho G_{k} \,dW_k=\mathcal{R}_\rho G_{k}(t)W_k(t)- \int_0^t \frac{d}{dr}\left( \mathcal{R}_\rho G_{k}(r)\right)\Bigr\rvert_{r = s} W_k(s)\,ds
        \\
        &\quad =\mathcal{R}_\rho G_{k}(t)W_k(t)
        + \frac{1}{\rho} \int_0^t  \mathcal{R}_\rho G_{k}(s) W_k(s)\,ds 
        - \frac 1 \rho \int_0^t  G_{k}(s) W_k(s)\,ds,
    \end{align*}
and similarly for $ \int_0^t \mathcal{R}_\rho G^n_{k} dW^n_k$. Thus, we have
\begin{align*}
&A_2(t) = \sum_{k=1}^N\left(\mathcal{R}_\rho G^n_{k}(t)W^n_k(t)-\mathcal{R}_\rho G_{k}(t)W_k(t)\right)-\frac 1 \rho \sum_{k=1}^N\int_0^t  (G^n_{k}(s) W^n_k(s)-G_{k}(s) W_k(s))\,ds\\
&\qquad +\frac{1}{\rho} \sum_{k=1}^N\int_0^t  (\mathcal{R}_\rho G^n_{k}(s) W^n_k(s)-\mathcal{R}_\rho G_{k}(s) W_k(s) )\,ds
:=A_{21}(t)+A_{22}(t)+A_{23}(t).
\end{align*}
Let us analyze these three terms. By the Cauchy-Schwartz inequality,
\begin{equation*}
    \begin{split}
        \left\Vert \mathcal{R}_\rho \Phi (t) \right\Vert_{\mathcal H} 
        &=  \left\Vert \int_0^t \frac{1}{\rho} e^{ - \frac{t -s}{\rho} }\Phi(s)\, ds \right\Vert_{\mathcal H} 
        \le \rho^{-1/2} \left\Vert \Phi \right\Vert_{L^2([0,T];\mathcal H)}.
    \end{split}
\end{equation*}
Then, we have for $A_{21}$
\begin{align*}
     &\Prob\biggl( \sup_{t\in[0,T]}\Vert A_{21}(t) \Vert_{\mathcal H} > \frac{\epsilon}{3}  \biggr)
       \leq 
        \sum_{k=1}^N\Prob\biggl( \sup_{t\in[0,T]} | W_k^n(t) | \left\Vert \mathcal{R}_\rho (G_{k}^n- G_{k})(t) \right\Vert_{\mathcal H} >\frac \epsilon {6N}  \biggr)
        \\&\quad \qquad\qquad
        +\sum_{k=1}^N\Prob\biggl( \sup_{t\in[0,T]} | W_k^n(t)-W_k(t) | \left\Vert \mathcal{R}_\rho G_{k}(t) \right\Vert_{\mathcal H}  >\frac \epsilon {6N}  \biggr)
        \\
        &\qquad\qquad\leq \sum_{k=1}^N\Prob\biggl( \sup_{t\in[0,T]} | W_k^n(t) | \left\Vert G_{k}^n- G_{k} \right\Vert_{L^2([0,T];\mathcal H)} >\frac {\epsilon \rho^{1/2} } {6N}  \biggr)
        \\
        &\quad\qquad\qquad+\sum_{k=1}^N\Prob\biggl( \sup_{t\in[0,T]} | W_k^n(t)-W_k(t) | \left\Vert G_{k} \right\Vert_{L^2([0,T];\mathcal H)}  >\frac {\epsilon \rho^{1/2} } {6N}  \biggr).
\end{align*}
Note that \eqref{eq:ConvAsmpNoise} implies, for every $k$, that $\sup_{t\in [0,T]}|W_k^n(t)-W_k(t)|\to 0$ in probability and that $\{\sup_{t\in [0,T]}|W_k^n(t)|\}_{n\in\mathbb{N}}$ is uniformly bounded in probability; similarly, \eqref{eq:ConvAsmpStochIntegrand} implies, for every $k$, that $\Vert G_{k}^n- G_{k} \Vert_{L^2([0,T];\mathcal H)}\to 0$ in probability.
Hence, we can choose $n_3$, depending on $\epsilon$ and $ \delta$ (as well as $N$ and $\rho$ which have been already fixed), such that each addend on the right-hand side is smaller than $\delta/N$. Hence, for every $n\ge n_3$, we get
\begin{align*}
    \Prob\biggl( \sup_{t\in[0,T]}\Vert A_{21}(t) \Vert_{\mathcal H} > \frac{\epsilon}{3}  \biggr)
       &\leq 2\delta.
\end{align*}
For $A_{23}$, we use \eqref{eq:smoothing_properties} and get
\begin{align*}
        &\Prob\biggl( \sup_{t\in[0,T]}\Vert A_{23}(t)\Vert_{\mathcal H} >\frac{\epsilon}{3}  \biggr)
        \leq \sum_{k=1}^N\Prob\biggl( \sup_{t\in[0,T]} | W_k^n(t) | \int_0^T \left\Vert \mathcal{R}_\rho (G_{k}^n- G_{k})(s) \right\Vert_{\mathcal H} \,ds >\frac {\epsilon\rho} {6N}  \biggr)
        \\
        &\quad\qquad\qquad+\sum_{k=1}^N\Prob\biggl( \sup_{t\in[0,T]} | W_k^n(t)-W_k(t) |  \int_0^T \left\Vert \mathcal{R}_\rho G_{k}(s) \right\Vert_{\mathcal H}  \,ds   >\frac {\epsilon\rho} {6N}  \biggr)
        \\
        &\qquad\qquad\leq \sum_{k=1}^N\Prob\biggl( \sup_{t\in[0,T]} | W_k^n(t) | T^{1/2}\left\Vert G_{k}^n- G_{k} \right\Vert_{L^2([0,T];\mathcal H)} >\frac {\epsilon \rho } {6N}  \biggr)
        \\
        &\quad\qquad\qquad+\sum_{k=1}^N\Prob\biggl( \sup_{t\in[0,T]} | W_k^n(t)-W_k(t) | T^{1/2} \left\Vert G_{k} \right\Vert_{L^2([0,T];\mathcal H)}  >\frac {\epsilon \rho } {6N}  \biggr).
\end{align*}
By the same argument as above, we can find $n_4$, depending on $\epsilon$, $\delta$, $N$, and $\rho$, such that each addend on the right-hand side of the above inequality is smaller than $\delta/N$. Hence, for every integer $n\ge n_4$, we get
\begin{align*}
    \Prob\biggl( \sup_{t\in[0,T]}\Vert A_{23}(t)\Vert_{\mathcal H} >\frac{\epsilon}{3} \biggr) \le 2\delta.
\end{align*}
We can deal with $A_{22}(t)$ in the same way. Then, there is $n_5\in\mathbb{N}$ such that 
\begin{align*}
    \Prob\biggl( \sup_{t\in[0,T]}\Vert A_{22}(t)\Vert_{\mathcal H} >\frac{\epsilon}{3}  \biggr) \le 2\delta
\end{align*}
whenever $n\ge n_5$. Putting all estimates together, we conclude that, for every integer $n$ that is greater than $n_0:=\max\{n_1,n_2,n_3,n_4,n_5\}$,
\begin{align*}
    \Prob \left( \sup_{t\in[0,T]}\left\Vert \mathcal{I}^n(t)- \mathcal{I}(t)  \right\Vert_{\mathcal H} > 5\epsilon\right)
    \le 16\delta.
\end{align*}
By the arbitrariness of $ \epsilon$ and $ \delta $, the lemma is proved.
\end{proof}

\printbibliography

\end{document}